\documentclass[11pt,oneside,reqno]{amsart}
\usepackage{geometry}
\usepackage[utf8]{inputenc}
\usepackage{amstext,latexsym,amsbsy,amsmath,amssymb,amsthm,mathtools,relsize,geometry,enumerate}
\usepackage{hyperref}
\hypersetup{
  colorlinks   = true, 
  urlcolor     = blue, 
  linkcolor    = red, 
  citecolor   = red 
}
\usepackage{mathtools,enumerate,enumitem}

\usepackage{mathrsfs}

\usepackage{comment}
\usepackage{graphicx}
\usepackage{epstopdf}
\setkeys{Gin}{width=\linewidth,totalheight=\textheight,keepaspectratio}
\graphicspath{{./images/}}

\usepackage{multicol}
\usepackage{booktabs}
\usepackage{mathrsfs}
\usepackage{color}

\newcommand{\rbb}{\mathbb{R}}
\newcommand{\zbb}{\mathbb{Z}}

\newcommand{\W}{\mathcal{W}}

\newcommand{\Pcal}{\mathcal{P}}

\newcommand{\la}{\langle}
\newcommand{\ra}{\rangle}

\renewcommand{\i}{\textup{i}}

\newcommand{\mi}{\wedge}

\renewcommand{\d}{\textup{d}}

\newcommand{\f}{\varphi}

\newcommand{\Fcal}{\mathcal{F}}
\newcommand{\E}{\mathbb{E}}

\renewcommand{\P}{\mathbb{P}}

\newcommand{\Law}{\textup{Law}}

\newcommand{\A}{\mathcal{A}}

\newcommand{\Tcal}{\mathcal{T}}

\newcommand{\TV}{\textup{TV}}

\theoremstyle{plain}
\newtheorem{theorem}{Theorem}[section]

\newtheorem{lemma}[theorem]{Lemma}

\newtheorem{proposition}[theorem]{Proposition}
\newtheorem{definition}[theorem]{Definition}

\theoremstyle{definition}

\newtheorem{remark}[theorem]{Remark}

\numberwithin{equation}{section}

 \title{Exponential mixing for the stochastic Kuramoto-Sivashinsky equation on the 1D torus}

\author{Peng Gao$^1$ and Hung D.~Nguyen$^2$}

\address{$^1$ School of Mathematics and Statistics, and Center for Mathematics
and Interdisciplinary Sciences, Northeast Normal University, Changchun, China}
\address{\hspace{0.15cm} Email: gaopengjilindaxue@126.com}

\address{$^2$ Department of Mathematics, University of Tennessee, Knoxville, Tennessee, USA}
\address{\hspace{0.15cm} Email: hnguye53@utk.edu}

\begin{document}

\begin{abstract}
    In this paper, we study the large-time behaviors of the Kuramoto-Sivashinsky equation (KSE) on the 1D torus while being subjected to random perturbation via additive Gaussian noise. It is well-known that under suitable assumptions on the stochastic forcing, the KSE admits a unique invariant probability measure. In this work, we make further progress on the topic of ergodicity by addressing the issue of convergence rate toward equilibrium. In comparison with the previous results, we can prove that the unique invariant probability measure is exponentially attractive and smallness condition of anti-diffusion coefficient is not necessary here. The proof relies on a coupling argument while making use of Lyapunov functions motivated by those of deterministic equations.
\end{abstract}

\maketitle



\section{Introduction} \label{sec:intro}

Letting $L>0$, we are interested in the long-time behaviors of the KSE on the one-dimensional periodic domain $[-L/2,L/2]$, under the impact of random perturbations. The equation is given by  
\begin{align} \label{eqn:KSE}
    \d u(t) + D^2 u(t)\d t +\gamma D^4 u(t)\d t + u(t) Du(t)\d t = \sigma\d W(t),\quad u(0)=u_0,
\end{align}
where $D^m=\frac{\partial^m}{\partial x^m}$, $\gamma>0$ is the diffusion constant, and $\sigma W$ is an additive Gaussian noise that is white in time and colored in space. The classical KSE (without noise) arises naturally in the modeling of phase turbulence in reaction-diffusion system  and plane flame propagation, originally developed independently in the pioneered work of \cite{kuramoto1978diffusion,kuramoto1975formation, kuramoto1976persistent} and \cite{michelson1977nonlinear,sivashinsky1977nonlinear}, respectively. Since then, there has been an extensive literature on both the deterministic and stochastic variations of the KSE (see Section \ref{sec:intro:literature} below). In particular, concerning the large-time asymptotics of equation \eqref{eqn:KSE}, it was established in \cite{weinan2002gibbsian, ferrario2008invariant} that under suitable assumptions on the noise term $\sigma W$, the dynamics admits a unique invariant probability measure in the $L^2-$space of periodic functions, albeit without a convergence speed. More recently, the setting of unbounded domains was explored in the work of \cite{gao2024polynomial}, and it can be shown that the solutions to a fourth order parabolic equation on the whole line are attracted toward equilibrium with power law decaying rate. Our goal of the paper is to extend the unique ergodicity results from \cite{ferrario2008invariant} and to bridge the gap on the mixing rate of \eqref{eqn:KSE} between the torus and the whole space. More specifically, the main theorem of the present article can be summarized as follows:
\begin{theorem} \label{thm:mixing:meta}
    Under appropriate assumptions on the stochastic forcing $\sigma W$, for all $\gamma>0$, equation \eqref{eqn:KSE} admits a unique invariant probability measure $\mu$. Moreover, there exists a positive constant $c$ such that for all initial condition $u_0\in L^2$ and suitable observables $f:L^2\to \rbb$, the following holds:
    \begin{align} \label{ineq:mixing:meta}
       \Big| \E f(u(t;u_0))- \int_{L^2}f(u)\mu(\d u)\Big| \le C e^{-c t},\quad t\ge 0,
    \end{align}
    for some positive constant $C=C(f,u_0)$ independent of $t$.
\end{theorem}

We refer the reader to Theorem \ref{thm:geometric-ergodicity} for a precise statement of Theorem \ref{thm:mixing:meta}. Before diving into the methodology used to establish Theorem \ref{thm:mixing:meta}, in Section \ref{sec:intro:literature} below, we will briefly review the related literature of equation \eqref{eqn:KSE}.

\subsection{Literature review} \label{sec:intro:literature}
Since its introduction in the 1970s, the KSE has been instrumental in the connection between PDE and dynamical systems. The well-posedness of the deterministic equation was established in the work of \cite{nicolaenko1985some,tadmor1986well}. It is also known that the KSE admits finitely many determining modes and a finite-dimensional global attractor \cite{nicolaenko1985some}, which is contained in an inertial manifold characterizing the long-time behavior of the solutions \cite{foias1986inertial}. Furthermore, while numerical evidences in \cite{fantuzzi2015construction,goluskin2019bounds,wittenberg1999scale} indicate that $\limsup_{t\to\infty}\|u(t)\|_{L^2}$ is insensitive to the length of the spatial intervals, justifying this limit rigorously remains an open challenge in the study of the asymptotic behaviors of the KSE. Analytical results in this direction appeared as early as in the work of \cite{nicolaenko1985some} for a special class of solutions. The technique of \cite{nicolaenko1985some} was then refined in \cite{collet1993global,goodman1994stability} to cover arbitrary solutions. Recent progress on optimal bounds dependent on $L$ has been made in \cite{giacomelli2005new,goldman2015new,otto2009optimal}. In particular, we note that the framework from \cite{collet1993global,goodman1994stability} motivates the construction of Lyapunov functions, which are one of the main crucial ingredients of proving Theorem \ref{thm:mixing:meta} in the present article. See Section \ref{sec:intro:methodology} for a further discussion of this point. It is important to note that in spite of an extensive development for the 1D KSE, the global well-posedness in higher dimensions is still not resolved completely \cite{larios2024remarks}.

The KSE under random perturbations has also received a lot of attention during the last several decades. Under different assumptions on the stochastic forcing, the well-posedness of \eqref{eqn:KSE} was investigated in \cite{cuerno1995renormalization,duan2001stochastic, ferrario2008invariant,karma1993competition}. On the one hand, there are many work exploring the solutions' properties in finite time windows. To mention a few examples, we refer the reader to the study of the associated Kolmogorov equation \cite{yang2012kolmogorov}, the controllability \cite{gao2018control,gao2020control}, the absolute continuity of the solutions' laws \cite{ferrario2008absolute}, the averaging principle \cite{gao2018averaging}, the large deviation with multiplicative noise \cite{rose2021large}, and the moderate deviation \cite{rose2022moderate}. On the other hand, as discussed elsewhere in \cite{ferrario2008invariant,kuramoto1978diffusion}, while chaotic behaviors are observed numerically temporarily, the presence of noise induces large-time stability through statistically steady states, e.g., the existence of random attractors \cite{duan1998dynamics,yang2006random,yang2007dynamics} and invariant probability measures \cite{ferrario2008invariant,gao2022irreducibility, gao2024polynomial}. 
In particular, the unique ergodicity of \eqref{eqn:KSE} was established as early as in \cite{weinan2002gibbsian} by adopting a strategy of \cite{weinan2001gibbsian}, in which a Gibbsian dynamics
for the noisy low-mode evolution is considered and whose high-mode dynamics is entirely characterized
by the historical behavior of the low modes. In addition, the work of \cite{ferrario2008invariant} studied ergodicity making use of the strong Feller property, which captures the instantaneous smoothing property of the corresponding Markov semigroup. We remark that this crucial fact relies on non-degenerate noise, i.e., the stochastic forcing is excited in every direction of the Fourier space. More recently, under the assumption that only a finite number of Fourier modes is randomly perturbed, the irreducibility of \eqref{eqn:KSE} was proven in \cite{gao2022irreducibility}. 
We note that despite a similarity between the noise structure of equation \eqref{eqn:KSE} and those in \cite{gao2024polynomial}, the main difference is that we restrict the equation to the torus, which allows us to construct strong Lyapunov functions, ultimately deducing the exponential mixing. In what follows, we will provide a brief description of the methodology used to formulate Theorem \ref{thm:mixing:meta}.

\subsection{Methodology in this work} \label{sec:intro:methodology} 

Turning back to equation \eqref{eqn:KSE}, we follow closely the coupling framework of \cite{butkovsky2014subgeometric,butkovsky2020generalized,hairer2011asymptotic,kulik2017ergodic,kulik2018generalized} to establish the exponential convergence rate stated in Theorem \ref{thm:mixing:meta}. The argument essentially consists of three main ingredients: a Lyapunov function, the contracting property of the Markov semigroup with respect to a distance-like function $d$, and the $d$-small property of bounded sets, cf. Definition \ref{def:contracting}. Notably, owing to the presence of the anti diffusion term $D^2$, the derivation of Lyapunov function is the main challenge that we face to the extend that it induces further complication when we proceed to derive contracting properties. More specifically, as typically found in other stochastic dynamics presented in \cite{butkovsky2020generalized, hairer2011asymptotic,kulik2018generalized}, one seeks to obtain an estimate of the form
\begin{align} \label{ineq:Lyapunov:meta}
     V(u(t)) + \int_0^t S(u(\ell)) \d \ell \le V(u(0))+ M(t)+ Ct, 
\end{align}
where $V(u)$ and $S(u)$ are energy-like functionals of the solution $u$, $M(t)$ is a semi Martingale process whose quadratic variation is controlled by $S(u)$, and $C$ is a positive constant independent of time $t$ and the initial condition. However, since the linear operator $D^2+\gamma D^4$ is not guaranteed to be strictly positive for arbitrary value of $\gamma$, a simple calculation on the torus $[-L/2,L/2]$ only produces 
\begin{align*}
    \frac{1}{2}\d \|u\|^2_{L^2}-\|Du\|^2_{L^2}\d t+\gamma\|D^2 u\|^2_{L^2}\d t=C\d t+ \d M .
\end{align*}
While this is sufficient to conclude the well-posedness, it does not provide the dissipative effect as in \eqref{ineq:Lyapunov:meta}, which is needed to deduce the mixing rate. To overcome the difficulty, we resort to the technique developed in \cite{collet1993global,goodman1994stability} dealing with the same issue for deterministic settings. As it turns out, one can shift the solution $u(t)$ by a periodic function $\varphi(t)$ (possibly depending on $u(t)$) to arrive at
\begin{align} \label{ineq:Lyapunov:meta:phi}
     V(u(t)-\f(t)) + \int_0^t S(u(\ell)-\f(\ell)) \d \ell \le V(u(0))+ M(t)+ Ct.
\end{align}
We note that \eqref{ineq:Lyapunov:meta:phi} is possible thanks to the torus domain together with the nonlinear structure and the diffusion $\gamma$ that we are able to leverage to handle the $D^2$ term. 
Upon using Cauchy inequality and periodicity, we may recover from \eqref{ineq:Lyapunov:meta:phi} a \eqref{ineq:Lyapunov:meta}-like estimate of the form
\begin{align*} 
     V(u(t)) + \int_0^t S(u(\ell)) \d \ell \le C_1 V(u(0))+ M(t)+ Ct,
\end{align*}
where a new constant $C_1$ now appears in front of the term $V(u(0))$. Unlike the framework of \cite{butkovsky2020generalized} requiring $C_1=1$, cf. \eqref{ineq:Lyapunov:meta}, in our setting, $C_1$ may be large and thus does not satisfy the criterion of Lyapunov function and contracting properties in \cite{butkovsky2020generalized}. To this end, we tackle the issue by modifying the argument of \cite{butkovsky2020generalized} and leveraging delicate exponential moment bounds. This allows us to successfully achieve the existence of a spectral gap with respect to a suitable Wasserstein distance. This is precisely stated in Theorem \ref{thm:geometric-ergodicity}, whose result implies Theorem \ref{thm:mixing:meta}. As mentioned in Section \ref{sec:intro:literature}, in comparison with previous literature, although we restrict equation \eqref{eqn:KSE} to the torus, we are able to establish the exponential mixing rate, which was not addressed in \cite{ferrario2008invariant}. All of this will be explained in detail in Section \ref{sec:mixing} where we carry out the proof of exponential mixing.

The rest of the paper is organized as follows: in Section \ref{sec:results}, we introduce the conditions on the noise structure and the functional settings of periodic domains that are needed for the analysis. We also state the main result on exponential mixing of \eqref{eqn:KSE} through Theorem \ref{thm:geometric-ergodicity}. In Section \ref{sec:moment-estimate}, we collect useful moment bounds which play the role of Lyapunov functions for \eqref{eqn:KSE}. In Section \ref{sec:mixing}, we present the proof of Theorem \ref{thm:geometric-ergodicity} in details making use of the estimates from Section \ref{sec:moment-estimate}. The paper concludes with two appendices: in Appendix \ref{sec:periodic_functions}, we consider periodic functions and establish crucial inequalities that are employed to construct Lyapunov functions for \eqref{eqn:KSE}. In Appendix \ref{sec:aux}, we collect useful estimates that are used to prove the main theorem.

\section{Main result} \label{sec:results}

\subsection{Functional settings} \label{sec:results:functional-setting}

Following \cite{collet1993global,ferrario2008invariant}, we consider the space $\Pcal_L$ of real $L$-periodic functions with vanishing integral, namely,
\begin{align*}
\Pcal_L=\Big\{u: u(x+L)=u(x),\,\int_{-L/2}^{L/2}u(x)\d x=0  \Big\}.
\end{align*}
Recall that given $u\in \Pcal_L$, we may formally recast $u$ using Fourier series representation 
\begin{align*}
    u=\frac{1}{\sqrt{L}}\sum_{k\in \zbb\setminus\{0\}}u_k e^{i\frac{2\pi k}{L}x},
\end{align*}
with $u_{-k}=\overline{u_k}$ (so that $u$ is real-valued). 

For every $m\ge0$, we denote by $H^m_L$ the Sobolev space of $L$-periodic functions endowed with the inner product
\begin{align*}
    \la u,v\ra_{H^m_L} = \Big(\frac{2\pi}{L}\Big)^{2m}\!\!\!\sum_{k\in \zbb\setminus\{0\}}|k|^{2m}u_kv_k,
\end{align*}
and the norm
\begin{align*}
    \|u\|^2_{H^m_L} = \Big(\frac{2\pi}{L}\Big)^{2m}\!\!\!\sum_{k\in \zbb\setminus\{0\}}|k|^{2m}|u_k|^2.
\end{align*}
That is,
\begin{align*}
H^m_L =\Big\{ u\in \Pcal_L: \|u\|_{H^m_L}<\infty  \Big\}.
\end{align*}
We note that when $m$ is a non-negative integer, by Parseval's identity, it holds that
\begin{align*}
    \la D^m u,D^m v\ra_{H^0_L} = \la u,v\ra_{H^m_L}.
\end{align*} 
Let $\{e_k\}_{k\ge 1}$ be an orthonormal basis in $H^0_L$ that diagonalizes the Laplacian, i.e.,
\begin{align} \label{eqn:D^2e_k=-alpha_ke_k}
    D^2e_k = -\alpha_k e_k,
\end{align}
where $\alpha_k\ge 0$ is diverging to infinity. For $N\ge 1$, we denote by $P_N$ the projection of $H^0_L$ onto span$\{e_1,\dots,e_N\}$, namely,
\begin{align*}
P_Nu = \sum_{k=1}^N \la u,e_k\ra_{H_L}e_k.
\end{align*}
If there is no confusion about the $L$-periodicity, we will drop the subscript $L$ and simply denote $H^m:=H^m_L$. Particularly, we will  use the notation $H:=H^0$.

Next, we turn to the stochastic forcing $\sigma W$ that appears on the right-hand side of \eqref{eqn:KSE}. We assume that $\sigma W$ is degenerate in the sense that it has the representation
\begin{align} \label{form:sigma.W}
    \sigma W(t) = \sum_{k=1}^M \sigma_k B_k(t),
\end{align}
where $M$ is a positive integer, $\{\sigma_k\}_{k=1,\dots,M}$ is a sequence of elements in $H$ and $\{B_k\}_{k=1,\dots,M}$ is a sequence of i.i.d. one-dimensional standard Brownian motions, each defined on the same stochastic basis $(\Omega,\Fcal,(\Fcal_t)_{t\ge 0},\P)$ satisfying the usual conditions \cite{karatzas2012brownian}. Following the framework of \cite{butkovsky2020generalized,glatt2021long,glatt2017unique, kulik2018generalized}, we may regard $\sigma$ as a linear bounded map from $\rbb^M$ to $H$, and fix a maximal integer $N$ such that
\begin{align} \label{cond:W}
    P_NH\subset \textup{Range}(\sigma) = \textup{Span}\{\sigma_1,\dots,\sigma_M\}.
\end{align}
It follows from the choice of $N$ that the inverse map $\sigma^{-1}:P_NH\to\rbb^M$ is also a bounded operator.

Having introduced the functional settings needed for the analysis, we briefly discuss the well-posedness of \eqref{eqn:KSE}. More specifically, we record the following result in \cite{ferrario2008invariant} giving the existence and uniqueness of weak solutions.
\begin{proposition}{\cite[Theorem 3.4]{ferrario2008invariant}} \label{prop:well-posed}
    Let $\sigma W$ be given as in \eqref{form:sigma.W}. Then, for all $\gamma>0$ and $u_0\in H$, the following holds:

1. There exists a unique stochastic process $u(t)=u(t;u_0)$ such that $u\in C([0,\infty);H)$ is $\Fcal_t$-adapted and that for every $T>0$ and $v\in H^2$, $\P$-a.s.
\begin{align*}
    \la u(t),v\ra_{H} &= \int_0^t\la Du(s),Dv\ra_{H}-\gamma\la D^2u(s),D^2v\ra_{H}+\frac{1}{2}\la u(s)^2,Dv\ra_{H}\d s\\
    &\qquad +\la u_0,v\ra_{H}+ \int_0^t \la v,\sigma\d W(s)\ra_{H},\quad \textup{a.e. }\, t\in[0,T].
\end{align*}

2. The solution $u(t)$ is continuous with respect to initial condition. That is, for all $t\ge 0$,
\begin{align*}
    \E\|u(t;u_0^n)-u(t;u_0)\|^2_{H}\to 0,\quad \textup{as }\,n\to\infty,
\end{align*}
whenever $\|u_0^n-u_0\|_{H}\to 0$, as $n\to\infty$.
\end{proposition}
The proof of Proposition \ref{prop:well-posed} can be carried out by employing an argument typically found in the settings of degenerate additive noise \cite{ ferrario2008invariant,flandoli1994dissipativity, glatt2021long}. Namely, for $a>0$, we consider the linear equation 
\begin{align*}
    \d z(t) + D^2 z(t)\d t +\gamma D^4 z(t)\d t + az(t)\d t = \sigma \d W(t),\quad z(0)=u_0.
\end{align*}
In the above, $a>0$ can be regarded as an auxiliary damping constant. Using the factorization method \cite{da2014stochastic}, it can be shown that for all $a$ sufficiently large, $z_a$ is a well-defined stochastic process taking values in $C([0,\infty);H)$. Then, by subtracting $z_a$ away from $u$, observe that the difference $v=u-z_a$ satisfies the following random PDE
\begin{align*}
    \frac{\d}{\d t} v(t) + D^2 v(t) +\gamma D^4 v(t) + (v(t)+z_a(t)\big) D\big(v(t)+z_a(t)\big) - az(t) = 0,\quad v(0)=0.
\end{align*}
Then, the well-posedness of $v$ can be derived using a classical Galerkin approach together with appropriate a priori bounds., Then, upon recovering $u=v+z_a$, we obtain the solutions for \eqref{eqn:KSE}, as well as the continuity with respect to initial conditions. We refer the reader to \cite[Section 3]{ferrario2008invariant} for a more detailed proof of Proposition \ref{prop:well-posed}.

As a consequence of the well-posedness result, we may introduce the Markov transition probabilities given by
\begin{align*}
    P_t(u_0,A) = \P(u(t;u_0)\in A),
\end{align*}
which are well-defined for all $u_0\in H$ and Borel set $A\subset H$. The associated Markov semigroup $P_t$ given by
\begin{align}
    P_t f(u_0) = \E f(u(t;u_0)),
\end{align}
is a mapping from $\mathcal{B}_b (H)$ to $\mathcal{B}_b(H)$ where $\mathcal{B}_b(H)$ denotes the class of bounded Borel measurable functions $f:H\to\rbb$. We note that thanks to the continuity with respect to initial data, cf. Proposition \ref{prop:well-posed} part 2, $P_t$ has the Feller property. That is, $P_tf\in C_b(H)$ whenever $ f\in C_b(H)$ where $C_b(H)$ denotes the set of bounded continuous functions $f:H\to\rbb$. 

Letting $\nu$ be an element in $\Pcal r(H)$, the collection of probability measures on $H$, the action of $P_t$ on $\nu$ is defined as
\begin{align*}
    P_t\nu(A) = \int_{H} P_t(u,A)\nu(\d u).
\end{align*}
We recall that $\nu$ is said to be invariant for $P_t$ if for all $t\ge 0$,
\begin{align*}
    P_t\nu=\nu.
\end{align*}
In the work of \cite{ferrario2008invariant}, under the hypothesis of non-degenerate stochastic forcing, i.e., noise is actively excited in every Fourier direction, it can be shown that $P_t$ admits a unique invariant measure \cite[Theorem 7.1]{ferrario2008invariant}. In particular, while the existence result relies on the usual Krylov-Bogoliubov procedure, the uniqueness result follows from establishing the irreducibility and the strong Feller properties, which relies heavily on the non-degeneracy nature of the noise.

\subsection{Exponential mixing} \label{sec:results:Exponential mixing} We now turn to the main topic of the paper on the exponential mixing of $P_t$ in the presence of degenerate noise as in \eqref{form:sigma.W}-\eqref{cond:W}. Following the framework of \cite{butkovsky2020generalized, hairer2011asymptotic}, we recall that a function $d:H\times H\to [0,\infty)$ is called \emph{distance-like} if it is symmetric, lower semicontinuous, and $d(u,v)=0$ if and only if $u=v$, cf. \cite[Definition 4.3]{hairer2011asymptotic}. The Wasserstein distance on $\Pcal r (H)$ associated with $d$ is denoted by $\W_d$ given by
\begin{align} \label{form:W}
    \W_d(\nu_1,\nu_2) = \inf \E d(X_1,X_2),
\end{align}
where the infimum runs over all pairs of random variables $(X_1,X_2)$ such that $X_1\sim\nu_1$ and $X_2\sim \nu_2$. The most well-known example of a distance-like function is the discrete metric $d(u,v) = \mathbf{1}\{u\neq v\}$, which corresponds to the total variation Wasserstein distance, denoted by $\W_{\TV}$. The second distance-like function that will also be helpful for the analysis is defined for 
$K,\beta>0$
\begin{align} \label{form:d_K}
    d_{K,\beta}(u,v) = K\theta_\beta(u,v) \mi K\theta_\beta(v,u) \mi 1.
\end{align}
where the function $\theta_\beta:H \times H \to[0,\infty)$ is given by
\begin{align} \label{form:theta(u,v)}
    \theta_\beta(u,v) = \|u-v\|_H e^{\beta \|u\|^2_H}.
\end{align}
 The distance-like function that we will actually employ to measure the convergence rate toward equilibrium is defined for $K,\beta>0$
 \begin{align} \label{form:dtilde_K,beta}
\tilde{d}_{K,\beta}(u,v)=\sqrt{d_{K,\beta}(u,v)\big(1+e^{\beta\|u\|^2_H}+e^{\beta\|v\|^2_H}\big)}
 \end{align}

 We now state the main result of the paper that gives the exponential mixing of $P_t$.

\begin{theorem} \label{thm:geometric-ergodicity}
    Let $\tilde{d}_{K,\beta}$ be the distance-like function defined in \eqref{form:dtilde_K,beta} and $N$ be the constant from expression \eqref{cond:W}. Then, for all $\beta$ sufficiently small, there exist positive constants $K$ and $N$ sufficiently large such that the following holds
    \begin{align}\label{ineq:geometric-ergodicity}
        \W_{\tilde{d}_{K,\beta}}(P_t\nu_1,P_t\nu_2) \le Ce^{-ct}\W_{\tilde{d}_{K,\beta}}(\nu_1,\nu_2),\quad t\ge T,\, \nu_1,\,\nu_2\in \Pcal r(H),
    \end{align}
    for some positive constants $C = c(\gamma,N,K,\beta),c=c(\gamma,N,K,\beta)$ and $T=T(\gamma,N,K,\beta)$ independent of $t$, $\nu_1$, and $\nu_2$. In the above, $\W_{\tilde{d}_{K,\beta}}$ is the Wasserstein distance associated with $\tilde{d}_{K,\beta}$ defined in \eqref{form:dtilde_K,beta}.
\end{theorem}

 In order to establish Theorem \ref{thm:geometric-ergodicity}, we will draw upon the framework of generalized coupling from \cite{butkovsky2020generalized} tailored to our setting. The first two main ingredients of the proof include proving that $d_{K,\beta}$ is contracting for $P_t$ and that bounded sets are $d_{K,\beta}$-small. See Definition \ref{def:contracting} below. Combining them with suitable Lyapunov bounds will then allow us to ultimately produce the spectral gap \eqref{ineq:geometric-ergodicity}. As mentioned in the introduction, owing to the presence of the anti-diffusion term $D^2$, we have to navigate the issue of Lyapunov functions. This will be addressed in Section \ref{sec:moment-estimate} using inequalities on periodic functions collected in Appendix \ref{sec:periodic_functions}. Moreover, as a result of a Lyapunov function designed specifically for $P_t$, the argument for contracting properties in our work is slightly different from \cite{butkovsky2020generalized}, leveraging delicate exponential moment bounds. The exact argument will be presented in Section \ref{sec:mixing} where we supply the proof of Theorem \ref{thm:geometric-ergodicity}.

\section{Moment estimates} \label{sec:moment-estimate}

Throughout the rest of the paper, $c$ and $C$ denote generic positive constants that may change from line to line. The main parameters that they depend on will appear between parenthesis, e.g., $c(T,q)$ is a function of $T$ and $q$.

In this section, we collect useful energy estimates on the solution $u(t;u_0)$ of \eqref{eqn:KSE} through Lemma \ref{lem:moment:int.|D^2u|^2} and Lemma \ref{lem:moment:H:exponential}. More specifically, in Lemma \ref{lem:moment:int.|D^2u|^2}, we perform path-wise estimates on the solutions while making use of the auxiliary results from Appendix \ref{sec:periodic_functions}. In lemma \ref{lem:moment:H:exponential}, we derive exponential moment bounds, which play the role of Lyapunov functions employed to establish Exponential mixing stated in Theorem \ref{thm:geometric-ergodicity}.

In order to precisely formulate the results, following the framework of \cite{collet1993global,nicolaenko1985some}, we introduce $\A_L$ denoting the space of odd $L$-periodic functions, namely,
\begin{align} \label{form:A_L}
    \A_L := \big\{u:u(x+L)=u(x),\, u(-x)=-u(x)\big\}.
\end{align}
 Let $\f\in \A_{2L}$ be the function as in Lemma \ref{lem:R(u)} with $\gamma/2$. That is, for all $u\in H$ and $b\in\rbb$, it holds that 
    \begin{align} \label{ineq:R(u)>|D^2u|^2:gamma/2}
       &\frac{1}{2}\gamma\|D^2 u\|^2_{H_{2L}}-\|Du\|^2_{H_{2L}}+\frac{1}{2}\la u^2,D\f(\cdot+b)\ra_{H_{2L}}  \notag  \\
       &\ge \frac{1}{8}\gamma\|D^2u\|^2_{H_{2L}} +\frac{1}{2}\|u\|^2_{H_{2L}}-\frac{1}{4L}\big|\la u,D\f(\cdot+b)\ra_{H_{2L}}  \big|^2.
    \end{align}
In the above, we recall $\la \cdot,\cdot\ra_{H_{2L}}$ is the inner product in $H_{2L}$
\begin{align*}
    \la u,v\ra_{H_{2L}} = \int_{-L}^L u(x)v(x)\d x.
\end{align*}
Also, for $b(\cdot)\in C^1([0,\infty);\rbb)$, we denote by $\f_{b(\cdot)}$ the translation of $\f$ defined as
    \begin{align} \label{form:phi_b(t)}
        \f_{b(t)}(x) = \f(x+b(t)) ,\quad x\in\rbb.
    \end{align}
In particular, given $u(t)$ the solution of \eqref{eqn:KSE}, we introduce the process $b(t)$ satisfying the equation
\begin{align}\label{eqn:b(t)}
    \frac{\d}{\d t}b(t) = \frac{1}{4L}\la u,D\f_{b(t)}\ra_{H_{2L}},\quad b(0)=0.
\end{align}
We note that since $u\in C([0,\infty);H)$, the solution $b(t)$ of \eqref{eqn:b(t)} is guaranteed to exist. See \cite[Appendix A]{collet1993global}.

In what follows, we state the first result of this section through Lemma \ref{lem:moment:int.|D^2u|^2} giving a pathwise estimate on the solution $u(t)$ in terms of the function $\f_{b(t)}$ defined in \eqref{ineq:R(u)>|D^2u|^2:gamma/2}-\eqref{form:phi_b(t)}-\eqref{eqn:b(t)}.

\begin{lemma} \label{lem:moment:int.|D^2u|^2}
For every $u_0\in H$, the following holds
    \begin{align} \label{ineq:moment:int.|D^2u|^2}
       & \|u(t)\|^2_{H} + \frac{1}{2}\gamma\int_0^t \|D^2u(s)\|^2_{H}\d s +\frac{1}{2}\int_0^t\|u(s)-\f_{b(s)}\|^2_{H_{2L}}\d s\notag \\
    & \le 4\|u_0\|^2_{H}+2\int_0^t\la u(s)-\f_{b(s)},\sigma\d W(s)\ra_{H_{2L}}+C_0t+C_1,\quad t\ge 0,
    \end{align}
    where $\f$ and $\f_{b(\cdot)}$ are the functions as in \eqref{ineq:R(u)>|D^2u|^2:gamma/2}-\eqref{form:phi_b(t)}-\eqref{eqn:b(t)}, and $C_0$ and $C_1$ are defined as
    \begin{align} \label{form:C_0}
    C_0 = \Big(1+\frac{2}{\gamma}\Big)\|\f\|^2_{H_{2L}}+2\gamma\|D^2\f\|^2_{H_{2L}} +\|\sigma\|^2_{H_{2L}},\quad
    C_1 = 3\|\f\|^2_{H_{2L}}.
\end{align}
\end{lemma}

\begin{proof} 
First of all, we note that $\f_{b(t)}$ satisfies
    \begin{align*}
        \frac{\d}{\d t}\f_{b(t)} = D\f(\cdot+b(t)) b'(t).
    \end{align*}
    Now, a routine calculation produces
    \begin{align} \label{eqn:d|u-Phi_b|^2}
        \frac{1}{2}\d\|u-\f_{b}\|^2_{H_{2L}} & = \la u-\f_b,-D^2u-\gamma D^4u-uDu -D\f_b b'\ra_{H_{2L}}\d t \notag \\
        &\qquad + \la u-\f_b,\sigma\d W\ra_{H_{2L}}+\frac{1}{2}\|\sigma\|^2_{H_{2L}}\d t.
    \end{align}
We employ the fact that $\f$ is $2L$-periodic while making use of integrations by parts to compute
\begin{align*}
    & \la u-\f_b,-D^2u-\gamma D^4u-uDu -D\f_b b'\ra_{H_{2L}}\\
    &= \|Du\|^2_{H_{2L}} -\gamma\|D^2u\|^2_{H_{2L}}  +\la uDu,\f_b\ra_{H_{2L}}\\
    &\qquad-\la u,D\f_b\ra_{H_{2L}}b'+\la \f_b,D\f_b\ra_{H_{2L}} b'+\la \f_b,D^2u\ra_{H_{2L}}+\gamma\la \f_b,D^4 u\ra_{H_{2L}}\\
    &=  \|Du\|^2_{H_{2L}} -\gamma\|D^2u\|^2_{H_{2L}}  -\frac{1}{2}\la u^2,D\f_b\ra_{H_{2L}}\\
    &\qquad -\la u,D\f_b\ra_{H_{2L}}b'+\la \f_b,D^2u\ra_{H_{2L}}+\gamma\la D^2\f_b,D^2 u\ra_{H_{2L}}.
\end{align*}
In the last identity above, we invoked the fact that $\f_b$ is $2L$-periodic. Next, we employ Cauchy-Schwarz inequality to infer
\begin{align*}
    &\la \f_b,D^2u\ra_{H_{2L}}+\gamma\la D^2\f_b,D^2 u\ra_{H_{2L}} \\
    &\le \frac{1}{\gamma}\|\f_b\|^2_{H_{2L}}+\gamma\|D^2\f_b\|^2_{H_{2L}}+\frac{\gamma}{2}\|D^2u\|^2_{H_{2L}},
\end{align*}
whence
\begin{align*}
    & \la u-\f_b,-D^2u-\gamma D^4u-uDu -D\f_b b'\ra_{H_{2L}}\\
    &\le \|Du\|^2_{H_{2L}} -\frac{1}{2}\gamma\|D^2u\|^2_{H_{2L}}  -\frac{1}{2}\la u^2,D\f_b\ra_{H_{2L}}\\
    &\qquad -\la u,D\f_b\ra_{H_{2L}}b'+\frac{1}{\gamma}\|\f_b\|^2_{H_{2L}}+\gamma\|D^2\f_b\|^2_{H_{2L}}.
\end{align*}
This together with \eqref{ineq:R(u)>|D^2u|^2:gamma/2} implies
\begin{align*}
    &\la u-\f_b,-D^2u-\gamma D^4u-uDu -D\f_b b'\ra_{H_{2L}}\\
    &\le - \frac{1}{8}\gamma\|D^2u\|^2_{H_{2L}} -\frac{1}{2}\|u\|^2_{H_{2L}}+\frac{1}{4L}\big|\la u,D\f_b\ra_{H_{2L}}  \big|^2\\
    &\qquad-\la u,D\f_b\ra_{H_{2L}}b'+\frac{1}{\gamma}\|\f_b\|^2_{H_{2L}}+\gamma\|D^2\f_b\|^2_{H_{2L}}.
\end{align*}
To further estimate the above right-hand side, let $b(\cdot)$ satisfy equation \eqref{eqn:b(t)}. We employ the elementary inequality $2(a^2+b^2)\ge (a+b)^2$, $a,b\in\rbb$, to see that
\begin{align} \label{ineq:<u-phi>}
    &\la u-\f_b,-D^2u-\gamma D^4u-uDu -D\f_b b'\ra_{H_{2L}} \notag \\
    &\le - \frac{1}{8}\gamma\|D^2u\|^2_{H_{2L}} -\frac{1}{2}\|u\|^2_{H_{2L}}+\frac{1}{\gamma}\|\f_b\|^2_{H_{2L}}+\gamma\|D^2\f_b\|^2_{H_{2L}} \notag \\
    &\le - \frac{1}{8}\gamma\|D^2u\|^2_{H_{2L}} -\frac{1}{4}\|u-\f_b\|^2_{H_{2L}}+\frac{1}{2}\|\f_b\|^2_{H_{2L}}+\frac{1}{\gamma}\|\f_b\|^2_{H_{2L}}+\gamma\|D^2\f_b\|^2_{H_{2L}}.
\end{align}
From \eqref{eqn:d|u-Phi_b|^2}, we get
\begin{align} \label{ineq:d|u-Phi_b|^2}
    \d\|u-\f_{b}\|^2_{H_{2L}} & \le - \frac{1}{4}\gamma\|D^2u\|^2_{H_{2L}}\d t -\frac{1}{2}\|u-\f_b\|^2_{H_{2L}}\d t+2\la u-\f_b,\sigma\d W\ra_{H_{2L}} \notag \\
    &\qquad+\Big[\Big(1+\frac{2}{\gamma}\Big)\|\f_b\|^2_{H_{2L}}+2\gamma\|D^2\f_b\|^2_{H_{2L}} +\|\sigma\|^2_{H_{2L}}\Big] \d t.
\end{align}
As a consequence, we integrate both sides with respect to time $t$ and obtain
\begin{align*}
    &\|u(t)-\f_{b(t)}\|^2_{H_{2L}} + \frac{1}{4}\gamma\int_0^t \|D^2u(s)\|^2_{H_{2L}}\d s +\frac{1}{2}\int_0^t\|u(s)-\f_{b(s)}\|^2_{H_{2L}}\d s\\
    & \le\|u_0-\f\|^2_{H_{2L}}+\Big[\Big(1+\frac{2}{\gamma}\Big)\|\f_b\|^2_{H_{2L}}+2\gamma\|D^2\f_b\|^2_{H_{2L}} +\|\sigma\|^2_{H_{2L}}\Big] t\\
    &\qquad +2\int_0^t\la u(s)-\f_{b(s)},\sigma\d W(s)\ra_{H_{2L}}.
\end{align*}
It follows that
    \begin{align*}
    &\frac{1}{2}\|u(t)\|^2_{H_{2L}} + \frac{1}{4}\gamma\int_0^t \|D^2u(s)\|^2_{H_{2L}}\d s +\frac{1}{2}\int_0^t\|u(s)-\f_{b(s)}\|^2_{H_{2L}}\d s\\
    & \le \|\f_{b(t)}\|^2_{H_{2L}}+2\|u_0\|^2_{H_{2L}}+2\|\f\|^2_{H_{2L}} +2\int_0^t\la u(s)-\f_{b(s)},\sigma\d W(s)\ra_{H_{2L}}\\
    &\qquad+\Big[\Big(1+\frac{2}{\gamma}\Big)\|\f_b\|^2_{H_{2L}}+2\gamma\|D^2\f_b\|^2_{H_{2L}} +\|\sigma\|^2_{H_{2L}}\Big] t.
\end{align*}
Since $u$ is $L$-periodic and $\f$ is $2L$-periodic, we have
\begin{align} \label{eqn:L-periodic}
    \|u\|^2_{H_{2L}} =2\|u\|^2_{H},\quad \|\f_{b(t)}\|^2_{H_{2L}}= \|\f(\cdot +b(t))\|^2_{H_{2L}} = \|\f\|^2_{H_{2L}}.
\end{align}
So, 
\begin{align*}
   & \|u(t)\|^2_{H} + \frac{1}{2}\gamma\int_0^t \|D^2u(s)\|^2_{H}\d s +\frac{1}{2}\int_0^t\|u(s)-\f_{b(s)}\|^2_{H_{2L}}\d s\\
    & \le 4\|u_0\|^2_{H}+2\int_0^t\la u(s)-\f_{b(s)},\sigma\d W(s)\ra_{H_{2L}}+C_0t+C_1,
\end{align*}
where $C_0$ and $C_1$ are defined in \eqref{form:C_0}. This establishes \eqref{ineq:moment:int.|D^2u|^2}, thereby concluding the proof.

\end{proof}

As a consequence of the proof of Lemma \ref{lem:moment:int.|D^2u|^2}, we provide two exponential moment bounds through Lemma \ref{lem:moment:H:exponential}, cf. \eqref{ineq:moment:H:exponential:int_0^t|D^2u|} and \eqref{ineq:moment:H:exponential:|u|^2_H}. The former will be employed to establish the contraction property of the Markov semigroup $P_t$ whereas the latter supplies the Lyapunov estimate in order to extract the convergence rates toward the unique invariant probability measure.

\begin{lemma}   \label{lem:moment:H:exponential}
For all $\beta>0$ sufficiently small and for all $u_0\in H$, the followings hold:
    \begin{align} \label{ineq:moment:H:exponential:int_0^t|D^2u|}
        \E\exp\Big\{ \frac{1}{2}\beta\gamma\int_0^t\|D^2u(s)\|^2_{H}\d s  \Big\}\le 2\exp\{4\beta \|u_0\|^2_H+\beta C_0 t+\beta C_1\},\quad t\ge 0,
    \end{align}
    and  
    \begin{align} \label{ineq:moment:H:exponential:|u|^2_H}
         \E\exp\big\{ \beta \|u(t)\|^2_H \big\} \le 2\exp\Big\{4\beta C_0+\beta C_1 + 4\beta e^{-\frac{1}{4}t}\|u_0\|^2_H \Big\},\quad t\ge 0.
    \end{align}
    In the above, $\f_{b(\cdot)}$ is the function defined in \eqref{ineq:R(u)>|D^2u|^2:gamma/2}-\eqref{form:phi_b(t)}-\eqref{eqn:b(t)}, and $C_0$, $C_1$ are defined in \eqref{form:C_0}.
\end{lemma}

\begin{proof}
    With regard to \eqref{ineq:moment:H:exponential:int_0^t|D^2u|}, we aim to employ the exponential Martingale inequality to establish the moment bound \eqref{ineq:moment:H:exponential:int_0^t|D^2u|}. To see this, letting $\beta>0$ be given and be chosen later, from \eqref{ineq:d|u-Phi_b|^2}, we get
\begin{align} \label{ineq:d.beta|u-phi_b|}
    \d\beta\|u-\f_{b}\|^2_{H_{2L}} & \le - \frac{1}{4}\beta\gamma\|D^2u\|^2_{H_{2L}}\d t -\frac{1}{2}\beta\|u-\f_b\|^2_{H_{2L}}\d t+\beta C_0\d t \notag \\
    &\qquad+2\beta\la u-\f_b,\sigma\d W\ra_{H_{2L}}.
\end{align}
Setting
\begin{align*}
    \d M = 2\beta\la u-\f_b,\sigma\d W\ra_{H_{2L}},
\end{align*}
observe that the corresponding quadratic variation process $\la M\ra(\cdot)$ satisfies
\begin{align*}
    \d\la M\ra = 4\beta^2 \big|  \la u-\f_b,\sigma\ra_{H_{2L}}\big|^2\d t\le 4\beta^2 \|\sigma\|^2_{H_{2L}} \|u-\f_b\|^2_{H_{2L}} \d t.
\end{align*}
Next, we recall the exponential Martingale inequality
\begin{align*}
    \P\Big( \sup_{t\ge 0}\Big[M(t)-\frac{1}{2}\lambda \la M\ra(t)\Big]\ge R\Big) \le e^{-\lambda R},\quad \lambda>0,\, R>0.
\end{align*}
Based on the above right hand side, we pick $\lambda=2$ and obtain
\begin{align*}
     \P\Big( \sup_{t\ge 0}\Big[M(t)-\la M\ra(t)\Big]\ge R\Big) \le e^{-2 R},
\end{align*}
implying
\begin{align} \label{ineq:E.exp.sup_[t>0]M(t)}
    \E \exp\Big\{ \sup_{t\ge 0}\Big[M(t)-\la M\ra(t)\Big] \Big\} \le 2.
\end{align}
So, provided $\beta$ is sufficiently small, e.g.,
\begin{align*}
    0<\beta < \frac{1}{8\|\sigma\|^2_{H_{2L}}},
\end{align*}
we deduce from \eqref{ineq:d.beta|u-phi_b|} that
\begin{align*}
    &\beta\|u(t)-\f_{b(t)}\|^2_{H_{2L}} + \frac{1}{4}\beta\gamma\int_0^t\|D^2u(s)\|^2_{H_{2L}}\d s -\beta C_0t\\
    &\le \beta\|u(t)-\f_{b(t)}\|^2_{H_{2L}} +  M(t)- \la M\ra(t),
\end{align*}
implying
\begin{align*}
    &\frac{1}{2}\beta\|u(t)\|^2_{H_{2L}} + \frac{1}{4}\beta\gamma\int_0^t\|D^2u(s)\|^2_{H_{2L}}\d s -\beta C_0 t-\beta \|\f_{b(t)}\|^2_{H_{2L}}\\
    &\le 2\beta\|u(0)\|^2_{H_{2L}} + 2\beta \|\f_{b(0)}\|^2_{H_{2L}}+  M(t)- \la M\ra(t).
\end{align*}
We employ \eqref{eqn:L-periodic} once again to infer 
\begin{align*}
    &\beta\|u(t)\|^2_{H} + \frac{1}{2}\beta\gamma\int_0^t\|D^2u(s)\|^2_{H}\d s -\beta C t-\beta \|\f\|^2_{H_{2L}}\\
    &\le 4\beta\|u(0)\|^2_{H} + 2\beta \|\f\|^2_{H_{2L}}+  M(t)- \la M\ra(t).
\end{align*}
In other words, it holds that
\begin{align*}
    &\beta\|u(t)\|^2_{H} + \frac{1}{2}\beta\gamma\int_0^t\|D^2u(s)\|^2_{H}\d s -\beta C_0 t-\beta C_1- 4\beta\|u(0)\|^2_{H} \\
    &\le  M(t)- \la M\ra(t) \le \sup_{s\ge0} \big[M(s)- \la M\ra(s)\big], 
\end{align*}
where $C_0,\,C_1$ are defined in \eqref{form:C_0}.
In light of \eqref{ineq:E.exp.sup_[t>0]M(t)}, we infer
\begin{align*}
    \E\exp\Big\{  \beta\|u(t)\|^2_{H} + \frac{1}{2}\beta\gamma\int_0^t\|D^2u(s)\|^2_{H}\d s\Big\}\le 2 \exp\Big\{ 4\beta\|u_0\|^2_{H} +\beta C_0 t+\beta C_1\Big\}.
\end{align*}
This establishes \eqref{ineq:moment:H:exponential:int_0^t|D^2u|}, as claimed.

Turning to the proof of \eqref{ineq:moment:H:exponential:|u|^2_H}, for $t\in[0,T]$, we have
\begin{align*}
    &\d \big[\beta e^{\frac{1}{4}(t-T)}\|u(t)-\f_{b(t)}\|^2_{H_{2L}}  \big]\\
    & = \beta\frac{1}{4}e^{\frac{1}{4}(t-T)}\|u(t)-\f_{b(t)}\|^2_{H_{2L}}\d t+ \beta e^{\frac{1}{4}(t-T)}\d\|u(t)-\f_{b(t)}\|^2_{H_{2L}}.
\end{align*}
From \eqref{ineq:d|u-Phi_b|^2}, we get
\begin{align} \label{ineq:d(e^(-t)|u-phi_b|)}
    &\d \big[\beta e^{\frac{1}{4}(t-T)}\|u(t)-\f_{b(t)}\|^2_{H_{2L}}  \big] \notag \\
    &\le  e^{\frac{1}{4}(t-T)}\Big[ -\frac{1}{4} \beta \|u(t)-\f_{b(t)}\|^2_{H_{2L}}\d t+2\beta\la u(t)-\f_{b(t)},\sigma\d W(t)\ra_{H_{2L}} +\beta C_0 \d t\Big].
\end{align}
Similar to the proof of \eqref{ineq:moment:H:exponential:int_0^t|D^2u|}, we aim to employ the exponential Martingale inequality to establish \eqref{ineq:moment:H:exponential:|u|^2_H} while making use of estimate \eqref{ineq:d(e^(-t)|u-phi_b|)}. To this end, let $M_1$ be the local Martingale process defined as
\begin{align*}
    \d M_1(t) = e^{\frac{1}{4}(t-T)}2\beta\la u(t)-\f_{b(t)},\sigma\d W(t)\ra_{H_{2L}}.
\end{align*}
Observe that for $t\in[0,T]$, the quadratic variation process $\la M_1\ra(t)$ satisfies
\begin{align*}
    \d \la M_1\ra(t) &= e^{\frac{1}{2}(t-T)}4\beta^2\big|\la u(t)-\f_{b(t)},\sigma\ra_{H_{2L}}\big|^2\d t\\
    &\le e^{\frac{1}{4}(t-T)}4\beta^2\|u(t)-\f_{b(t)}\|^2_{H_{2L}}\|\sigma\|^2_{H_{2L}}\d t.
\end{align*}
So, provided $\beta$ is small enough, e.g.,
\begin{align*}
    \beta \le \frac{1}{16 \|\sigma\|^2_{H_{2L}}},
\end{align*}
it holds that
\begin{align*}
     \d \la M_1\ra(t)  \le \frac{1}{4} \beta e^{\frac{1}{4}(t-T)}\|u(t)-\f_{b(t)}\|^2_{H_{2L}}\d t,\quad 0\le t\le T.
\end{align*}
As a consequence, \eqref{ineq:d(e^(-t)|u-phi_b|)} implies that
\begin{align*}
    \d \big[\beta e^{\frac{1}{4}(t-T)}\|u(t)-\f_{b(t)}\|^2_{H_{2L}}  \big]&\le \beta C_0 e^{\frac{1}{4}(t-T)}\d t + \d M_1(t)-\d\la M_1\ra(t).
\end{align*}
We integrate the above estimate with respect to $t\in[0,T]$ to obtain
\begin{align*}
    \beta \|u(T)-\f_{b(T)}\|^2_{H_{2L}} - \beta e^{-\frac{1}{4}T}\|u(0)-\f_{b(0)}\|^2_{H_{2L}}  - \beta C_0\int_0^T e^{\frac{1}{4}(t-T)}\d t \le M_1(T)-\la M_1\ra(T).
\end{align*}
In light of the exponential Martingale inequality \eqref{ineq:E.exp.sup_[t>0]M(t)} applying to $M_1$, we deduce
\begin{align*}
    &\E \exp \big\{\beta \|u(T)-\f_{b(T)}\|^2_{H_{2L}} - \beta e^{-\frac{1}{4}T}\|u(0)-\f_{b(0)}\|^2_{H_{2L}}  - 4\beta C_0\big\} \\
    &\le \E\exp\Big\{\sup_{0\le t\le T} \big[M_1(t)-\la M_1\ra(t)\big]\Big\} \le 2,
\end{align*}
whence
\begin{align*}
    \E \exp \big\{\beta \|u(T)-\f_{b(T)}\|^2_{H_{2L}}\big\}\le 2\exp\big\{ \beta e^{-\frac{1}{4}T}\|u(0)-\f_{b(0)}\|^2_{H_{2L}}  +4\beta C_0\big\}.
\end{align*}
It follows that 
\begin{align*}
    \E \exp \Big\{\frac{1}{2}\beta \|u(T)\|^2_{H_{2L}}\Big\}\le 2\exp\big\{\beta \|\f_{b(T)}\|^2_{H_{2L}}+ 2\beta e^{-\frac{1}{4}T}\|u_0\|^2_{H_{2L}}+ 2\beta e^{-\frac{1}{4}T}\|\f_{b(0)}\|^2_{H_{2L}}  +4\beta C_0\big\}.
\end{align*}
Recalling the choice of $\f$ from \eqref{ineq:R(u)>|D^2u|^2:gamma/2}-\eqref{form:phi_b(t)}-\eqref{eqn:b(t)}, and $C_0$, $C_1$ from \eqref{form:C_0}, we invoke once again the identity \eqref{eqn:L-periodic} to arrive at \eqref{ineq:moment:H:exponential:|u|^2_H}, namely,
\begin{align*}
    \E \exp \big\{\beta \|u(T)\|^2_{H}\big\}\le 2\exp\big\{4\beta e^{-\frac{1}{4}T}\|u_0\|^2_{H}+   4\beta C_0+\beta C_1\big\}.
\end{align*}
The proof is thus complete.
\end{proof}

\section{Proof of main result} \label{sec:mixing}

In this section, following the framework of \cite{butkovsky2020generalized, hairer2011asymptotic, kulik2017ergodic}, we proceed to establish Theorem \ref{thm:geometric-ergodicity} giving the Exponential mixing of \eqref{eqn:KSE}. In Section \ref{sec:mixing:proof-main-thm}, we review the main ingredients and conclude Theorem \ref{thm:geometric-ergodicity} while making use of useful auxiliary lemmas. In Section \ref{sec:mixing:proof-aux-results}, we provide the proofs of these auxiliary results, leveraging the estimates from Section \ref{sec:moment-estimate} and Appendix \ref{sec:periodic_functions}.

\subsection{Proof of Theorem \ref{thm:geometric-ergodicity}} \label{sec:mixing:proof-main-thm}
As mentioned in the Introduction, the argument for exponential mixing relies heavily on three crucial ingredients: a Lyapunov function, the contracting property of $P_t$ with respect to the distance-like function $d$ and the $d$-small property. For the reader's convenience, we recall these notions below in Definition \ref{def:contracting}.

\begin{definition} \label{def:contracting}
    1. A function $V:H\to[0,\infty)$ is called a Lyapunov function for $P_t$ if $V(u)\to\infty$ whenever $\|u\|_H\to\infty$ and that $V$ satisfies the estimate
    \begin{align*}
        \E V(u(t)) \le Ce^{-ct}V(u_0)+C,\quad t\ge 0,\, u_0\in H,
    \end{align*}
    for some positive constants $C,c$ independent of $t$ and $u_0$.

    2. \textup{(\cite[Definition 4.6]{hairer2011asymptotic})} A distance-like function $d$ bounded by 1 is called \textup{contracting} for $P_t$ if there exists $\alpha\in(0,1)$ such that for any $u,v\in H$ with $d(u,v)<1$, it holds that
    \begin{align*} 
        \W_{d}\big( P_t(u;\cdot), P_t(v;\cdot) \big) \le \alpha d(u,v).
    \end{align*}

    3. \textup{(\cite[Definition 4.4]{hairer2011asymptotic})} A set $B\subset H$ is called $d$-\textup{small} for $P_t$ if for some $\tilde{\alpha}\in(0,1)$,
    \begin{align*} 
       \sup_{u,v\in B} \W_{d}\big( P_t(u;\cdot), P_t(v;\cdot) \big) \le 1-\tilde{\alpha}.
    \end{align*}
\end{definition}

We now state the auxiliary results through Lemmas \ref{lem:Lyapunov}, \ref{lem:contracting} and \ref{lem:d-small} that will be employed to prove Theorem \ref{thm:geometric-ergodicity}, 

\begin{lemma} \label{lem:Lyapunov}
    For all $\beta>0$ sufficiently small, the function $V(u)=e^{\beta\|u\|^2_H}$ is a Lyapunov function for $P_t$. That is there exist positive constants $C$ and $c$ and $t_0$ such that the following holds
    \begin{align} \label{ineq:Lyapunov}
        \E \,e^{\beta\|u(t)\|^2_H} \le Ce^{-ct}e^{\beta\|u_0\|^2_H}+C,\quad t\ge t_0,
    \end{align}
    for all initial conditions $u_0\in H$.
\end{lemma}

\begin{lemma} \label{lem:contracting}
    Let $\beta$ be the same constant from Lemma \ref{lem:Lyapunov}, $N$ be the constant as in \eqref{cond:W} and $d_{K,\beta}$ be the distance-like function defined in \eqref{form:d_K}. Assume that $N$ is sufficiently large satisfying
    \begin{align} \label{cond:N}
     \frac{\gamma}{2}\alpha_N^2 > \frac{1}{2\gamma}+\frac{1}{2}\Big(4\frac{L}{\beta\gamma}+\frac{1}{8}\beta C_0\Big),
\end{align}
where $\alpha_N$ is as in \eqref{eqn:D^2e_k=-alpha_ke_k} and $C_0$ is the constant defined in \eqref{form:C_0}. Then, there exist positive constants $t_1$, $K_1$ such that for all $t\ge t_1$, $K\ge K_1$, $d_{K,\beta}$ is contracting for $P_t$ as in Definition \ref{def:contracting}, part 2. That is, there exists a positive constant $\alpha=\alpha(N,t,K,\beta)\in(0,1)$ such that
    \begin{align} \label{ineq:d-contracting}
        \W_{d_{K,\beta}}\big( P_t(u;\cdot), P_t(v;\cdot) \big) \le \alpha d_{K,\beta}(u,v),\quad u,v\in H,
    \end{align}
    whenever $d_{K,\beta}(u,v)<1$.
\end{lemma}

\begin{lemma} \label{lem:d-small}
     Let $\beta$ be the same constant from Lemma \ref{lem:Lyapunov}, $N$ be the constant satisfying \eqref{cond:W} and \eqref{cond:N}, and $d_{K,\beta}$ be the distance-like function defined in \eqref{form:d_K}. Then, there exist positive constants $t_2$, $K_2$ such that for all $t\ge t_2$, $K\ge K_2$, the set
     \begin{align*}
         B_R = \{u\in H: \|u\|_H\le R\},\quad R>0,
     \end{align*}
     is $d_{K,\beta}$-small for $P_t$ as in Definition \ref{def:contracting}, part 3. That is, there exists a positive constant $\tilde{\alpha}=\tilde{\alpha}(N,t,K,\beta,R)$ such that 
      \begin{align} \label{ineq:d-small}
       \sup_{u,v\in B_R} \W_{d_{K,\beta}}\big( P_t(u;\cdot), P_t(v;\cdot) \big) \le 1-\tilde{\alpha}.
    \end{align}
\end{lemma}

For the sake of clarity, we refer the proofs of Lemmas \ref{lem:Lyapunov}, \ref{lem:contracting} and \ref{lem:d-small} to Section \ref{sec:mixing:proof-aux-results}. Assuming the results of these lemmas, we are now in a position to conclude the proof of Theorem \ref{thm:geometric-ergodicity} by verifying the hypothesis of \cite[Theorem 4.8]{hairer2011asymptotic}.

\begin{proof}[Proof of Theorem \ref{thm:geometric-ergodicity}]
    Indeed, first of all, for every $\beta$ sufficiently small, Lemma \ref{lem:Lyapunov} supplies the Lyapunov bound for $P_t$ with $V(u)=e^{\beta\|u\|^2_H}$ being the Lyapunov function. Next, for all $N$ large (depending on $\beta$), let $t_i,K_i$, $i=1,2$ be the constants from Lemmas \ref{lem:contracting} and \ref{lem:d-small}, respectively. We note that for all $t\ge \max\{t_1,t_2\}$, $K\ge \max\{K_1,K_2\}$, $d_{K,\beta}$ defined in \eqref{form:d_K} is contracting and that every bounded ball $B_R$ is $d_{K,\beta}$-small for $P_t$. In view of \cite[Theorem 4.8]{hairer2011asymptotic}, we conclude the exponential convergence rate \eqref{ineq:geometric-ergodicity} with respect the distance-like function $\tilde{d}_{K,\beta}$ defined in \eqref{form:dtilde_K,beta}. This finishes the proof of Theorem \ref{thm:geometric-ergodicity}.
\end{proof}
  We also recall that Theorem  \ref{thm:geometric-ergodicity} is simply a more precise version
of Theorem \ref{thm:mixing:meta}, so that Theorem \ref{thm:mixing:meta} has now been proved.

\subsection{Proofs of auxiliary results} \label{sec:mixing:proof-aux-results}

In this subsection, we proceed to establish the auxiliary results presented in Section \ref{sec:mixing:proof-main-thm} that we employed to prove Theorem \ref{thm:geometric-ergodicity}. 

We start with Lemma \ref{lem:Lyapunov} whose proof is a short argument making use of Lemma \ref{lem:moment:H:exponential}, part 2 and Lemma \ref{lem:e^u<e^(-t+u)}.

\begin{proof}[Proof of Lemma \ref{lem:Lyapunov}] For all $\beta$ sufficiently small, we recall from Lemma \ref{lem:moment:H:exponential}, part 2, cf. \eqref{ineq:moment:H:exponential:|u|^2_H}, that
    \begin{align*}
    \E \exp \big\{\beta \|u(t)\|^2_{H}\big\}\le 2\exp\big\{4\beta e^{-\frac{1}{4}t}\|u_0\|^2_{H}+   4\beta C_0+\beta C_1\big\},
\end{align*}
where $C_0$ and $C_1$ are defined in \eqref{form:C_0}. It is clear that the function $f(t)=\beta\|u(t)\|^2_H$ satisfies condition \eqref{cond:e^u<e^(-t+u)}. It follows from Lemma \ref{lem:e^u<e^(-t+u)}, cf. \eqref{ineq:e^u<e^(-t+u)}, that we may infer $t_0$ such that
\begin{align*}
     \E \exp \big\{\beta \|u(t)\|^2_{H}\big\} \le Ce^{-ct}\exp\big\{\beta \|u_0\|^2_H\big\}+C,\quad t\ge t_0,
\end{align*}
for some positive constants $c=c(\beta)$, $C=C(\beta)$ independent of $t$ and $u_0$. This finishes the proof.
\end{proof}

Next, we turn our attention to the contracting and the $d$-small properties stated in Lemma \ref{lem:contracting} and Lemma \ref{lem:d-small}, respectively. Observe that inequality \eqref{ineq:d-contracting} is essentially an estimate on the difference between two solutions $u(t;u_0)$ and $u(t;v_0)$, $u_0,v_0\in H$. However, we will not directly compare them, owing to the presence of the nonlinearity in equation \eqref{eqn:KSE}. Instead, following the framework of \cite{butkovsky2014subgeometric,butkovsky2020generalized, glatt2017unique, glatt2021mixing,kulik2017ergodic}, we will employ a generalized coupling argument by slightly modifying equation \eqref{eqn:KSE} as follows: letting $N>0$ be as in \eqref{cond:W} and $\lambda>0$ be given, we introduce the process $v(t)=v(t;v_0,u_0)$ solving the equation
\begin{align} \label{eqn:KSE:v-coupling}
    \d v(t) + D^2 v(t)\d t +\gamma D^4 v(t)\d t + v(t) Dv(t)\d t = \sigma\d W(t)+\lambda P_N(u(t;u_0)-v(t)),
\end{align}
with the initial condition $v(0)=v_0$. We note that equation \eqref{eqn:KSE:v-coupling} only differs from \eqref{eqn:KSE} (with initial condition $v_0$) by the appearance of the additional term $\lambda P_N(u(t;u_0)-v(t))$, the motivation of which is two-fold. Firstly, shifting the original equation in suitable directions allows $u(t;u_0)$ and $v(t;v_0,u_0)$ to stay in close proximity as time $t$ tends to infinity, provided that noise is excited in a sufficient number of Fourier modes. This is captured in Lemma \ref{lem:coupling}. In turn, the synchronization can be exploited to effectively control the total variation distance between the distributions of $v(t;v_0,u_0)$ and $u(t;v_0)$. This is summarized in Lemma \ref{lem:error-in-law}. Altogether, the results of Lemmas \ref{lem:coupling} and \ref{lem:error-in-law} will be leveraged to establish the required contracting properties of Lemma  \ref{lem:contracting} and \ref{lem:d-small}.

Now, we state and prove Lemma \ref{lem:coupling}, giving a path-wise estimate on the difference of $u(t;u_0)$ and $v(t;v_0,u_0)$ in $H$-norm. The result of Lemma \ref{lem:coupling} will be particularly employed to prove Lemma \ref{lem:contracting}, producing the desired contracting property for $P_t$.

\begin{lemma} \label{lem:coupling}
    Given $u_0,v_0\in H$, let $u(t)=u(t;u_0)$ and $v(t)=v(t;v_0,u_0)$ respectively be the solutions of \eqref{eqn:KSE} and \eqref{eqn:KSE:v-coupling}. For all $C_2>0$, there exist $N=N(C_2,\gamma)$ and $\lambda=\lambda(C_2,\gamma)$ such that a.s.
    \begin{align} \label{ineq:|u-v|^2<|u_0-v_0|^2.e^(-ct+int.|Du|_infty)}
    \|u(t)-v(t)\|^2_H \le \|u_0-v_0\|^2_H \exp\Big\{ -C_2t  + \sqrt{L}\int_0^t \|D^2u(s)\|_{H}\d s\Big\}.
\end{align}
\end{lemma}
\begin{proof}

Setting $\bar{u}=u-v$, from \eqref{eqn:KSE} and \eqref{eqn:KSE:v-coupling}, observe that $\bar{u}$ satisfies the equation
\begin{align*}
   \frac{1}{2} \frac{\d}{\d t} \bar{u}(t) + D^2 \bar{u} (t) +\gamma D^4 \bar{u}(t) + u(t) Du(t)-v(t)Dv(t) = -\lambda P_N \bar{u}(t).
\end{align*}
A routine calculation on $H$-norm produces
\begin{align*}
    \frac{\d}{\d t} \|\bar{u}\|^2_H + \la D^2 \bar{u},\bar{u}\ra_H +\gamma \| D^2 \bar{u}\|^2_H + \la u Du-vDv,\bar{u}\ra_H = -\lambda\| P_N \bar{u}\|^2_H.
\end{align*}
On the one hand, we invoke Cauchy-Schwarz inequality to infer
\begin{align*}
    \la D^2 \bar{u},\bar{u}\ra_H\le \frac{\gamma}{2}\|D^2\bar{u}\|^2_H+\frac{1}{2\gamma}\|u\|^2_H.
\end{align*}
On the other hand, we employ an integration by parts to deduce that
\begin{align*}
    \la u Du-vDv,\bar{u}\ra_H &= \la \bar{u}^2,Du\ra_H-\la \bar{u}^2,D\bar{u}\ra_H+\la u,\bar{u}D\bar{u}\ra_H\\
    &=  \la \bar{u}^2,Du\ra_H-\frac{1}{2}\la Du,\bar{u}^2\ra_H\\
    &\le \frac{1}{2}\|\bar{u}\|^2_H\|Du\|_{L^\infty}.
\end{align*}
So, 
\begin{align*}
    \frac{1}{2} \frac{\d}{\d t}\|\bar{u}\|^2_H \le - \frac{\gamma}{2} \| D^2 \bar{u}\|^2_H -\lambda\| P_N \bar{u}\|^2_H+ \frac{1}{2\gamma}\|\bar{u}\|^2_H+\frac{1}{2}\|Du\|_{L^\infty}\|\bar{u}\|^2_H.
\end{align*}
Fixing $C_2>0$ be arbitrarily, let $N$ be large enough such that 
\begin{align} \label{ineq:alpha_N>C}
    \frac{\gamma}{2}\alpha_N^2\ge \frac{1}{2\gamma}+\frac{1}{2}C_2,
\end{align}
where $\alpha_N$ is as in \eqref{eqn:D^2e_k=-alpha_ke_k}. We pick
\begin{align} \label{form:lambda}
    \lambda = \frac{1}{2\gamma}+\frac{1}{2}C_2.
\end{align}
Observe that
\begin{align*}
    - \frac{\gamma}{2} \| D^2 \bar{u}\|^2_H -\lambda\| P_N \bar{u}\|^2_H & \le - \frac{\gamma}{2} \|(I-P_N) D^2 \bar{u}\|^2_H -\lambda\| P_N \bar{u}\|^2_H \\
    &\le -\frac{\gamma}{2} \alpha_N^2\|(I-P_N) \bar{u}\|^2_H -\lambda\| P_N \bar{u}\|^2_H \\
    &\le -\Big(\frac{1}{2\gamma}+\frac{1}{2}C_2\Big)\|\bar{u}\|^2_H.
\end{align*}
As a consequence,
\begin{align*}
    \frac{\d}{\d t}\|\bar{u}\|^2_H \le - C_2 \|\bar{u}\|^2_H+\|Du\|_{L^\infty}\|\bar{u}\|^2_H,
\end{align*}
whence
\begin{align*}
    \|\bar{u}(t)\|^2_H \le \|\bar{u}(0)\|^2_H \exp\Big\{ -C_2 t  +\int_0^t \|Du(s)\|_{L^\infty}\d s\Big\}.
\end{align*}
This together with \eqref{ineq:|u|_infty<L|Du|}, namely, 
\begin{align*}
   \|Du\|_{L^\infty} \le \sqrt{L}\|D^2u\|_H,
\end{align*}
immediately produces \eqref{ineq:|u-v|^2<|u_0-v_0|^2.e^(-ct+int.|Du|_infty)}, as claimed.
\end{proof}

Next, we formulate Lemma \ref{lem:error-in-law} providing an estimate on the total variation distance between $\Law\big(v(t;v_0,u_0)\big)$ and $P_t(v_0;\cdot)$ in terms of the difference $\|u(t;u_0)-v(t;v_0,u_0)\|_H$. The result of Lemma \ref{lem:error-in-law} will play a crucial role in establishing Lemma \ref{lem:d-small} on the $d$-small property.

\begin{lemma} \label{lem:error-in-law}
    Let $N$ be the constant as in \eqref{cond:W} and for every $u_0,v_0\in H$, let $u(t)=u(t;u_0)$ and $v(t)=v(t;v_0,u_0)$ be, respectively, the solutions of \eqref{eqn:KSE} and \eqref{eqn:KSE:v-coupling}. Then, the followings hold:
    \begin{align} \label{ineq:error-in-law:W_TV(v,u)<|u-v|^2_H}
    &\W_{\TV}\big( \Law(v(t;v_0,u_0)),P_t(v_0,\cdot)  \big)\notag \\
    &\le \frac{1}{2}\|\sigma^{-1}\|_{L(P_NH)}\lambda\sqrt{\E\int_0^t\|u(s;u_0)-v(s;v_0,u_0)\|^2_H\d s}.
    \end{align}
   and that
     \begin{align} \label{ineq:error-in-law:W_TV(v,u)<1-epsilon}
    &\W_{\TV}\big( \Law(v(t;v_0,u_0)),P_t(v_0,\cdot)  \big) \notag \\ &\le 1-\frac{1}{2}\exp\Big\{-\frac{1}{2} \|\sigma^{-1}\|^2_{L(P_NH)}\lambda^2  \E\int_0^t\|u(s;u_0)-v(s;v_0,u_0)\|^2_H\d s\Big\} .
    \end{align}
\end{lemma}

The proof of Lemma \ref{lem:error-in-law} is a slight reworking of that of \cite[Lemma 5.5]{glatt2022short}. See also \cite[Section 4.2]{butkovsky2020generalized}.

\begin{proof}[Proof of Lemma \ref{lem:error-in-law}] For notation convenience, we denote by $z_{[0,t]}$ the sample path of the process $z(\cdot)$ on $[0,t]$.

    With regard to \eqref{ineq:error-in-law:W_TV(v,u)<|u-v|^2_H}, let $\xi$ be the process defined as
    \begin{align} \label{form:xi(t)}
        \d\xi(t)=\d W(t) + \sigma^{-1}\big(\lambda P_N(u(t;u_0)-v(t;v_0,u_0))\big)\d t.
    \end{align}
    We recall that since $\sigma$ satisfies \eqref{cond:W}, $\sigma^{-1}$ is indeed a bounded map from $P_NH$ to $\rbb^M$. In particular, this means that $\xi$ is a well-defined process. Moreover, we see that $v(t;v_0,u_0)$ obeys the equation
    \begin{align} \label{eqn:KSE:v-coupling:xi}
    \d v(t) + D^2 v(t)\d t +\gamma D^4 v(t)\d t + v(t) Dv(t)\d t = \sigma\d \xi(t),\quad v(0)=v_0.
\end{align}
By the uniqueness of weak solutions, 
\begin{align*}
    \big\{ W_{[0,t]}=\xi_{[0,t]} \big\}\subset \big\{ u(s;v_0) = v(s;v_0,u_0),\, s\in[0,t] \big\}.
\end{align*}
In other words,
\begin{align*}
    \mathbf{1}\big\{ W_{[0,t]}\neq \xi_{[0,t]} \big\} \ge \mathbf{1}\big\{ \exists \, s\in[0,t], \, u(s;v_0)\neq v(s;v_0,u_0)\big\}.
\end{align*}
So, by the definition of total variation distance $\W_{\TV}$, we obtain
\begin{align} \label{ineq:W_TV(u,v)<W_TV(W,xi)}
    \W_{\TV}\big( \Law(v(t;v_0,u_0)),P_t(v_0,\cdot)  \big) \le \W_{\TV}\big(\Law\big(W_{[0,t]} \big),\Law\big(\xi_{[0,t]} \big) \big).
\end{align}
At this point, we recall Pinsker inequality that
\begin{align*}
   \W_{\TV}\big(\Law\big(W_{[0,t]} \big),\Law\big(\xi_{[0,t]} \big) \big) \le \sqrt{\frac{1}{2} D_{KL}\big(\Law\big(W_{[0,t]} \big),\Law\big(\xi_{[0,t]} \big)\big) },
\end{align*}
where $D_{KL}$ denotes the Kullback-Leibler divergence. Also, thanks to the expression of $\xi$ in \eqref{form:xi(t)}, \cite[Theorem A.2]{butkovsky2020generalized} implies that
\begin{align} \label{ineq:D_KL<|u-v|^2_H}
    D_{KL}\big(\Law\big(W_{[0,t]} \big),\Law\big(\xi_{[0,t]} \big)\big) \le \frac{1}{2}\E\int_0^t \|\sigma^{-1}\big(\lambda P_N(u(s;u_0)-v(s;v_0)\big)\|^2_H\d s.
\end{align}
Altogether, we get
\begin{align*}
    \W_{\TV}\big( \Law(v(t;v_0,u_0)),P_t(v_0,\cdot)  \big)  & \le \frac{1}{2}\sqrt{\E\int_0^t \|\sigma^{-1}\big(\lambda P_N(u(s;u_0)-v(s;v_0,u_0)\big)\|^2_H\d s}\\
    &\le \frac{1}{2}\|\sigma^{-1}\|_{L(P_NH)}\lambda\sqrt{\E\int_0^t \|u(s;u_0)-v(s;v_0,u_0)\|^2_H\d s}.
\end{align*}
This establishes \eqref{ineq:error-in-law:W_TV(v,u)<|u-v|^2_H}, as claimed.

Turning to \eqref{ineq:error-in-law:W_TV(v,u)<1-epsilon}, we recall \cite[inequality (A.2)]{butkovsky2020generalized} that
\begin{align*}
    \W_{\TV}\big(\Law\big(W_{[0,t]} \big),\Law\big(\xi_{[0,t]} \big) \big)  \le 1-\frac{1}{2}\exp\Big\{-D_{KL}\big(\Law\big(W_{[0,t]} \big),\Law\big(\xi_{[0,t]} \big)\big) \Big\}.
\end{align*}
This together with \eqref{ineq:D_KL<|u-v|^2_H} yields the bound
\begin{align*}
    &\W_{\TV}\big(\Law\big(W_{[0,t]} \big),\Law\big(\xi_{[0,t]} \big) \big)  \\
    &\le 1-\frac{1}{2}\exp\Big\{-\frac{1}{2}\E\int_0^t \|\sigma^{-1}\big(\lambda P_N(u(s;u_0)-v(s;v_0,u_0)\big)\|^2_H\d s\Big\}\\
    &\le 1-\frac{1}{2}\exp\Big\{-\frac{1}{2} \|\sigma^{-1}\|^2_{L(P_NH)}\lambda^2  \E\int_0^t\|u(s;u_0)-v(s;v_0,u_0)\|^2_H\d s\Big\} .
\end{align*}
In turn, we deduce from \eqref{ineq:W_TV(u,v)<W_TV(W,xi)} that
\begin{align*}
     &\W_{\TV}\big( \Law(v(t;v_0)),P_t(v_0,\cdot)  \big)\\ &\le 1-\frac{1}{2}\exp\Big\{-\frac{1}{2} \|\sigma^{-1}\|^2_{L(P_NH)}\lambda^2  \E\int_0^t\|u(s;u_0)-v(s;v_0,u_0)\|^2_H\d s\Big\} .
\end{align*}
The proof is thus finished.
\end{proof}

To conclude Lemmas \ref{lem:contracting} and \ref{lem:d-small}, we will need the following intermediate results.

\begin{lemma} \label{lem:A}
     Given $u_0,v_0\in H$, let $u(t)=u(t;u_0)$ and $v(t)=v(t;v_0,u_0)$ respectively be the solutions of \eqref{eqn:KSE} and \eqref{eqn:KSE:v-coupling}. Let $\beta$ be the same constant from Lemma \ref{lem:Lyapunov}, $N$ be the constant satisfying \eqref{cond:W} and \eqref{cond:N}, and $\theta_\beta$ be defined in \eqref{form:theta(u,v)}. Then, the followings hold:
    
    1. There exists a positive constant $C=C(N,\beta,\gamma)$ such that\begin{align} \label{ineq:A:W_TV(v,u)<theta(u,v)}
    \W_{\TV}\big( \Law(v(t)),P_t(v_0,\cdot)  \big)\le C\,\theta_\beta(u_0,v_0),\quad t\ge 0,
    \end{align}

    2. There exist a positive constant $T=T(N,\beta,\gamma)$ and a function $r_1:\rbb^+\to\rbb^+$ satisfying $\lim_{t\to\infty}r_1(t)\to 0$ such that
    and
    \begin{align} \label{ineq:A:theta(v,u)<e^(-ct)theta(u_0,v_0)}
        \E\, \theta_\beta (u(t),v(t))\le r_1(t)\theta_{\beta/2}(u_0,v_0),\quad t\ge T.
    \end{align}

\end{lemma}

\begin{lemma} \label{lem:B2}
 Given $u_0,v_0\in H$, let $u(t)=u(t;u_0)$ and $v(t)=v(t;v_0,u_0)$ respectively be the solutions of \eqref{eqn:KSE} and \eqref{eqn:KSE:v-coupling}. Let $\beta$ be the same constant from Lemma \ref{lem:Lyapunov}, $N$ be the constant satisfying \eqref{cond:W} and \eqref{cond:N}, and $\theta_\beta$ be defined in \eqref{form:theta(u,v)}. Then, the followings hold for all $R>0$:

 1. There exists a positive constant $\varepsilon=\varepsilon(N,\beta,R)>0$ such that
 \begin{align} \label{ineq:B2:W_TV<1-epsilon}
     \sup_{u_0,v_0\in B_R}\W_{\TV} \big( \Law(v(t)),P_t(v_0,\cdot)  \big) \le 1-\varepsilon,\quad t\ge 0.
 \end{align}

 2. 
 \begin{align} \label{ineq:B2:E.theta(u,v)<e^(-ct)}
      \sup_{u_0,v_0\in B_R}\E\, \theta_\beta(u(t),v(t)) \le 2R e^{\frac{1}{2}\beta R}r_1(t),\quad t\ge T,
 \end{align}
  where $T$ and $r_1(t)$ are respectively the time constant and the function from Lemma \ref{lem:A}, part 2.
\end{lemma}

\begin{remark} \label{rem:A-B2} 1. We note that Lemma \ref{lem:A} and Lemma \ref{lem:B2} respectively correspond to \cite[Assumption A]{butkovsky2020generalized} and \cite[Assumption B2]{butkovsky2020generalized}. The former is an ingredient needed to establish the contracting property stated in Lemma \ref{lem:contracting}, whereas the latter is crucial to deduce the $d$-small property in Lemma \ref{lem:d-small}. In order to verify \cite[Assumptions A and B2]{butkovsky2020generalized}, one can draw upon the argument of \cite[Lemma 4.3]{butkovsky2020generalized} where it is typically sufficient to derive a path-wise bound of the form 
\begin{align}\label{cond:A}
    U(X(t)) + \int_0^t S(X(s))\d s \le U(X(0))+ct+ M(t),
\end{align}
where $U$ and $S$ are functionals of the process $X(t)$, and $M$ is a semi Martingale process, whose quadratic variation can be controlled by $S$. See \cite[Assumption H2]{butkovsky2020generalized}. This condition \eqref{cond:A} is shown to be the case for a variety of SPDE examples discussed in \cite[Section 4]{butkovsky2020generalized} as well as the dynamics in \cite{glatt2022short, nguyen2023ergodicity}. In contrast, from Lemma \ref{lem:moment:int.|D^2u|^2}, if we set 
\begin{align*}
    U(u(t))=\|u\|^2_H,\quad S(u(t)) = \frac{1}{2}\gamma\|D^2u(t)\|^2_H+\frac{1}{2}\|u(t)-\varphi_{b(t)}\|^2_{H_{2L}},
\end{align*}
then estimate \eqref{ineq:moment:int.|D^2u|^2} can be recast as
\begin{align*}
     U(u(t)) + \int_0^t S(u(s))\d s \le 4U(u(0))+C_0t+C_1+ M(t).
\end{align*}
In particular, this does not verify \eqref{cond:A} due to the factor of $4$ in front of $U(u(0))$. As a consequence, we have to modify the proof of \cite[Lemma 4.3]{butkovsky2020generalized} tailored to the KSE setting making use of the exponential moment bounds in Lemma \ref{lem:moment:H:exponential}.
 In literature, it is worth noting that an analogous modification of the arguments from \cite{butkovsky2020generalized} was developed in \cite{glatt2025long} within a general framework, by also exploiting exponential Lyapunov structures.

2. We remark that estimate \eqref{ineq:A:theta(v,u)<e^(-ct)theta(u_0,v_0)} is slightly different from analogous results in literature where the function $\theta_\beta$ also appears on the right hand side instead of $\theta_{\beta/2}$ as in \eqref{ineq:A:theta(v,u)<e^(-ct)theta(u_0,v_0)}. While this modification is interesting on its own, \eqref{ineq:A:theta(v,u)<e^(-ct)theta(u_0,v_0)} will be needed later when we invoke the generalized triangle inequality \eqref{ineq:W_(d_K,beta):triangle} to establish the proof of Lemma \ref{lem:contracting}. 
    
\end{remark}

For the sake of clarity, we will defer the proofs of Lemmas \ref{lem:A} and \ref{lem:B2} to the end of this subsection. Assuming that the results of Lemmas \ref{lem:A} and \ref{lem:B2} hold, let us now conclude Lemmas \ref{lem:contracting} and \ref{lem:d-small}, whose arguments are relatively short and follow closely those of \cite[Lemma 6.4]{nguyen2023ergodicity}.

\begin{proof}[Proof of Lemma \ref{lem:contracting}]
Fix $u_0,v_0\in H$ such that $d_{K,\beta}(u,v)<1$. Without loss of generality, we may assume that $\|u_0\|_H\le \|v_0\|_H$. This implies that
\begin{align*}
    d_{K,\beta}(u_0,v_0) = K\theta_{\beta}(u_0,v_0)\wedge K\theta_{\beta}(v_0,u_0) \wedge 1 & = K\|u_0-v_0\|_H  e^{\beta\|u_0\|^2_H} \wedge  K\|u_0-v_0\|_H  e^{\beta\|v_0\|^2_H} \\
    & = K\|u_0-v_0\|_H e^{\beta\|u_0\|^2_H}\\
    &= K\theta_\beta(u_0,v_0).
\end{align*}
Let $u(t)=u(t;u_0)$ and $v(t)=v(t;v_0,u_0)$ respectively be the solution of \eqref{eqn:KSE} and \eqref{eqn:KSE:v-coupling}. In light of Lemma \ref{lem:W_(d_K,beta):triangle}, cf. \eqref{ineq:W_(d_K,beta):triangle}, we have
\begin{align} \label{ineq:W_(d_K,beta)(P_t(u_0),P_t(v_0))}
    &\W_{d_{K,\beta}}\big( P_t(u_0;\cdot), P_t(v_0;\cdot) \big) \\
    &\le e^{2\beta/K^2}  \big[\W_{d_{K,2\beta}}\big( P_t(u_0;\cdot), \Law(v(t)) \big) +  \W_{d_{K,2\beta}}\big(\Law(v(t)), P_t(v_0;\cdot) \big)\big].
\end{align}
On the one hand, we invoke Lemma \ref{lem:A}, part 2, to infer a positive constant $T=T(N,\beta,\gamma)$ and $r_1(t)$ satisfying $\lim_{t\to \infty}r_1(t)=0$ such that the following holds for all $t\ge T$.
\begin{align*}
    \W_{d_{K,2\beta}}\big( P_t(u_0;\cdot), \Law(v(t)) \big) \le \E \, d_{K,2\beta}(u(t),v(t)) \le K\E\,\theta_{2\beta}(u(t),v(t)) \le Kr_1(t)\theta_{\beta}(u_0,v_0).
\end{align*}
On the other hand, in view of Lemma \ref{lem:W_(d_K,beta)<W_TV} and Lemma \ref{lem:A}, part 1, we have
\begin{align*}
    \W_{d_{K,2\beta}}\big(\Law(v(t)), P_t(v_0;\cdot) \big) \le \W_{\TV}\big(\Law(v(t)), P_t(v_0;\cdot) \big)\le C\theta_\beta(u_0,v_0),\quad t\ge 0.
\end{align*}
Altogether, we deduce for $t\ge T$
\begin{align*}
    \W_{d_{K,\beta}}\big( P_t(u_0;\cdot), P_t(v_0;\cdot) \big) 
    &\le e^{2\beta/K^2}  \big[Kr_1(t)\theta_{\beta}(u_0,v_0) +   C\theta_\beta(u_0,v_0)\big]\\
    &= e^{2\beta/K^2}\Big[r_1(t)+\frac{C}{K}\Big]K\theta_\beta(u_0,v_0)\\
    &= e^{2\beta/K^2}\Big[r_1(t)+\frac{C}{K}\Big]d_{K,\beta}(u_0,v_0).
\end{align*}
Since $r_1$ and $C$ are both independent of $K$, we may take $t$ and $K$ sufficiently large such that
\begin{align*}
    \alpha:=e^{2\beta/K^2}\Big[r_1(t)+\frac{C}{K}\Big] <1,
\end{align*}
implying 
\begin{align*}
     \W_{d_{K,\beta}}\big( P_t(u_0;\cdot), P_t(v_0;\cdot) \big)  \le \alpha d_{K,\beta}(u_0,v_0).
\end{align*}
This establishes \eqref{ineq:d-contracting}, thereby finishing the proof.
\end{proof}

\begin{proof}[Proof of Lemma \ref{lem:d-small}]
Similar to the argument of Lemma \ref{lem:contracting}, for all $u_0,v_0\in B_R$, we employ Lemma \ref{lem:B2}, cf. \eqref{ineq:B2:E.theta(u,v)<e^(-ct)}, to infer
\begin{align*}
    \W_{d_{K,2\beta}}\big( P_t(u_0;\cdot), \Law(v(t)) \big) \le \E \, d_{K,2\beta}(u(t),v(t))&\le  K\E\,\theta_{2\beta}(u(t),v(t))\\
    &\le 2KRe^{\beta R}r_1(t), \quad t\ge T,
\end{align*}
where $T$ and $r_1$ are as in Lemma \ref{lem:A}, part 2. Also, estimate \eqref{ineq:B2:W_TV<1-epsilon} and Lemma \ref{lem:W_(d_K,beta)<W_TV} imply that
\begin{align*}
    \W_{d_{K,2\beta}}\big(\Law(v(t)), P_t(v_0;\cdot) \big) \le \W_{\TV}\big(\Law(v(t)), P_t(v_0;\cdot) \big) \le 1-\varepsilon,\quad t\ge 0,
\end{align*}
for some positive $\varepsilon>0$. Together with \eqref{ineq:W_(d_K,beta)(P_t(u_0),P_t(v_0))}, we get
\begin{align*}
     &\W_{d_{K,\beta}}\big( P_t(u_0;\cdot), P_t(v_0;\cdot) \big) \\
    &\le e^{2\beta/K^2}  \big[\W_{d_{K,2\beta}}\big( P_t(u_0;\cdot), \Law(v(t)) \big) +  \W_{d_{K,2\beta}}\big(\Law(v(t)), P_t(v_0;\cdot) \big)\big]\\
    &\le e^{2\beta/K^2}  \big[2KRe^{\beta R}r_1(t)+1-\varepsilon],
\end{align*}
It is important to note that $K$ is independent of $\varepsilon,\beta$ and $R$ whereas $r_1$ defined in \eqref{form:r_1(t)} does not depend on $\varepsilon$ and $R$. This allows us to take $K$ and $t$ sufficiently large to infer the existence of a positive constant $\tilde{\alpha}$ such that 
\begin{align*}
    e^{2\beta/K^2}  \big[2Ke^{\frac{1}{2}\beta R}r_1(t)+1-\varepsilon]=1-\tilde{\alpha},
\end{align*}
whence
\begin{align*}
    \sup_{u_0,v_0\in B_R}\W_{d_{K,\beta}}\big( P_t(u_0;\cdot), P_t(v_0;\cdot) \big) \le 1-\tilde{\alpha},
\end{align*}
    as claimed.
\end{proof}

\begin{remark}
     We remark that the proofs of Lemmas \ref{lem:contracting} and \ref{lem:d-small} are similar to that of \cite[Theorem 2.4]{butkovsky2020generalized}. The main difference is that in our approach, we opt for leveraging direct estimates on Wasserstein distances collected in Appendix \ref{sec:aux}. This allows for the proofs to be more self-contained, without having to go through the technique of optimal coupling, which is employed in \cite[Theorem 2.4]{butkovsky2020generalized}.
\end{remark}
Now, we present the proofs of Lemma \ref{lem:A} by exploiting the auxiliary results from Lemmas \ref{lem:contracting} and \ref{lem:error-in-law}.

\begin{proof}[Proof of Lemma \ref{lem:A}]
    1. With regard to \eqref{ineq:A:W_TV(v,u)<theta(u,v)}, we employ Lemma \ref{lem:error-in-law}, estimate \eqref{ineq:error-in-law:W_TV(v,u)<|u-v|^2_H} to infer that
    \begin{align}  \label{ineq:A:W_TV(v,u)<theta(u,v):a}
    &\W_{\TV}\big( \Law(v(t;v_0)),P_t(v_0,\cdot)  \big) \notag \\
    &\le \frac{1}{2}\|\sigma^{-1}\|_{L(P_NH)}\lambda\Big(\E\int_0^t \|u(s;u_0)-v(s;v_0)\|^2_H\d s\Big)^{1/2},
    \end{align}
    while Lemma \ref{lem:coupling}, estimate \eqref{ineq:|u-v|^2<|u_0-v_0|^2.e^(-ct+int.|Du|_infty)} implies that
    \begin{align*}
    &\E\int_0^t \|u(s;u_0)-v(s;v_0)\|^2_H\d s\\    
    &\le \|u_0-v_0\|_H^2\int_0^t\E \exp\Big\{-C_2 s + \sqrt{L}\int_0^s \|D^2u(\ell;u_0)\|_H\d \ell\Big\}\d s.
    \end{align*}
Using Cauchy-Schwarz inequality, we have 
\begin{align} \label{ineq:L.int.|D^2u|_H:Cauchy}
      \sqrt{L}\int_0^s \|D^2u(\ell;u_0)\|_H\d \ell \le  4\frac{L}{\beta\gamma}s+\frac{1}{16}\beta\gamma\int_0^s \|D^2u(s;u_0)\|_H^2\d \ell.
\end{align}
It follows from Lemma \ref{lem:moment:H:exponential}, estimate \eqref{ineq:moment:H:exponential:int_0^t|D^2u|} that for all $\beta$ small enough,
\begin{align*}
    & \E \exp\Big\{-C_2 s + \sqrt{L}\int_0^s \|D^2u(\ell;u_0)\|_H\d \ell\Big\}\\
    &\le 2\exp\Big\{-C_2 s +4\frac{L}{\beta\gamma}s+\frac{1}{8}\beta C_0 s+\frac{1}{8}\beta C_1+\frac{1}{2}\beta\|u_0\|^2_H   \Big\},
\end{align*}
where $C_0$ and $C_1$ are the constants defined in \eqref{form:C_0}. So,
\begin{align*}
    &\E\int_0^t \|u(s;u_0)-v(s;v_0)\|^2_H\d s\\    
    &\le\|u_0-v_0\|_H^2 \times 2 \exp\Big\{\frac{1}{8}\beta C_1+\frac{1}{2}\beta\|u_0\|^2_H\Big\}\int_0^t\exp\Big\{ -C_2 s +4\frac{L}{\beta\gamma}s+\frac{1}{8}\beta C_0 s \Big\}\d s.
    \end{align*}
Recalling the condition on $N$ from \eqref{ineq:alpha_N>C}, we observe that if $N$ satisfies \eqref{cond:W} and \eqref{cond:N}, i.e.,
\begin{align*}
     \frac{\gamma}{2}\alpha_N^2 > \frac{1}{2\gamma}+\frac{1}{2}\Big(4\frac{L}{\beta\gamma}+\frac{1}{8}\beta C_0\Big),
\end{align*}
we may infer the existence of a positive constant $C_2$ such that 
\begin{align} 
    C_2  > 4\frac{L}{\beta\gamma}+\frac{1}{8}\beta C_0,
\end{align}
and as a consequence
\begin{align} \label{ineq:E.int_0^t|u-v|^2_H}
    \E\int_0^t \|u(s;u_0)-v(s;v_0)\|^2_H\d s
        &\le \|u_0-v_0\|_H^2 \times 2 \frac{\exp\Big\{\frac{1}{8}\beta C_1+\frac{1}{2}\beta\|u_0\|^2_H\Big\}}{C_2  - 4\frac{L}{\beta\gamma}-\frac{1}{8}\beta C_0}.
\end{align}
From \eqref{ineq:A:W_TV(v,u)<theta(u,v):a}, we get
\begin{align*}
    & \W_{\TV}\big( \Law(v(t;v_0)),P_t(v_0,\cdot)  \big) \le \frac{\sqrt{2}\|\sigma^{-1}\|_{L(P_NH)}\lambda\, e^{\frac{1}{16}\beta C_1}}{\sqrt{C_2  - 4\frac{L}{\beta\gamma}-\frac{1}{8}\beta C_0} }\|u_0-v_0\|_H e^{\frac{1}{4}\beta\|u_0\|^2_H}.
\end{align*}
Recalling $\theta_\beta(u_0,v_0)=\|u_0-v_0\|_H e^{\beta\|u_0\|^2_H}$ from \eqref{form:theta(u,v)}, we deduce
\begin{align*}
    \W_{\TV}\big( \Law(v(t;v_0)),P_t(v_0,\cdot)  \big) &\le \frac{\sqrt{2}\|\sigma^{-1}\|_{L(P_NH)}\lambda\, e^{\frac{1}{16}\beta C_1}}{\sqrt{C_2  - 4\frac{L}{\beta\gamma}-\frac{1}{8}\beta C_0} }\|u_0-v_0\|_H e^{\beta\|u_0\|^2_H}\\
    &= \frac{\sqrt{2}\|\sigma^{-1}\|_{L(P_NH)}\lambda\, e^{\frac{1}{16}\beta C_1}}{\sqrt{C_2  - 4\frac{L}{\beta\gamma}-\frac{1}{8}\beta C_0} }\theta_\beta(u_0,v_0).
\end{align*}
This produces \eqref{ineq:A:W_TV(v,u)<theta(u,v)}, as claimed.

2. Concerning \eqref{ineq:A:theta(v,u)<e^(-ct)theta(u_0,v_0)}, we employ \eqref{ineq:|u-v|^2<|u_0-v_0|^2.e^(-ct+int.|Du|_infty)} and \eqref{ineq:L.int.|D^2u|_H:Cauchy} once again to infer
\begin{align*}
    \|u(t)-v(t)\|_H &\le \|u_0-v_0\|_H \exp\Big\{ -\frac{1}{2}C_2 t +\frac{\sqrt{L}}{2}\int_0^t \|D^2u(s)\|_H\d s\Big\}\\
    &\le \|u_0-v_0\|_H \exp\Big\{ -\frac{1}{2}C_2 t + 2\frac{L}{\beta \gamma}t +\frac{1}{32}\beta\gamma\int_0^t \|D^2u(s)\|^2_H\d s\Big\}.
\end{align*}
So, we employ Holder's inequality to obtain
\begin{align*}
    &\E\, \theta_\beta(u(t),v(t))\\
    & = \E \Big[\|u(t)-v(t)\|_H e^{\beta\|u(t)\|^2_H}\Big]\\
    &\le \|u_0-v_0 \|_H \exp\Big\{-\frac{1}{2}C_2 t+2\frac{L}{\beta\gamma}t \Big\}\sqrt{\E\exp\Big\{\frac{1}{16}\beta\gamma\int_0^t\|D^2 u(s)\|^2_H\d s \Big\}\,\E\exp\Big\{2\beta\|u(t)\|^2_H\Big\} }.
\end{align*}
On the one hand, from Lemma \ref{lem:moment:H:exponential}, estimate  \eqref{ineq:moment:H:exponential:int_0^t|D^2u|}, we get
\begin{align*}
  \E\exp\Big\{\frac{1}{16}\beta\gamma\int_0^t\|D^2u(s)\|^2_H\d s \Big\} \le 2 \exp\Big\{\frac{1}{2}\beta\|u_0\|^2_H+\frac{1}{8}\beta C_0 t+\frac{1}{8}\beta C_1\Big\}  .
\end{align*}
On the other hand, estimate \eqref{ineq:moment:H:exponential:|u|^2_H} in Lemma \ref{lem:moment:H:exponential} implies that
\begin{align*}
   \E\exp\Big\{2\beta\|u(t)\|^2_H\Big\}\le 2 \exp\Big\{8\beta C_0+2 \beta C_1+ 8 \beta e^{-\frac{1}{4}t}\|u_0\|^2_H \Big\}.
\end{align*}
It follows that
\begin{align*}
    & \E\, \theta_\beta(u(t),v(t))\\
    &\le 2\|u_0-v_0 \|_H \exp\Big\{-\frac{1}{2}C_2 t+2\frac{L}{\beta\gamma}t \Big\}\sqrt{\exp\Big\{\frac{1}{2}\beta\|u_0\|^2_H+\frac{1}{8}\beta C_0 t+\frac{1}{8}\beta C_1\Big\} }\\
    &\qquad\qquad \times \sqrt{\exp\Big\{8\beta C_0+2 \beta C_1+ 8 \beta e^{-\frac{1}{4}t}\|u_0\|^2_H \Big\}}.
\end{align*}
By taking $T$ large enough such that
\begin{align} \label{cond:t}
    8 e^{-\frac{1}{4}T} \le \frac{1}{2},
\end{align}
we obtain for all $t\ge T$
\begin{align*}
     \E\, \theta_\beta(u(t),v(t))
    &\le  \theta_{\beta/2}(u_0,v_0)\,r_1(t),
\end{align*}
where 
\begin{align} \label{form:r_1(t)}
    r_1(t): = 2 \exp\Big\{-\frac{1}{2}\Big(C_2 -4\frac{L}{\beta\gamma}- \frac{1}{8}\beta C_0\Big)t\Big\}\exp\Big\{4\beta C_0+\frac{17}{16} \beta C_1 \Big\}.
\end{align}
This establishes \eqref{ineq:A:theta(v,u)<e^(-ct)theta(u_0,v_0)}, thereby concluding the proof.
\end{proof}

Lastly, we provide the proof of Lemma \ref{lem:B2}, which also exploits the results of Lemma \ref{lem:error-in-law}. 

\begin{proof}[Proof of Lemma \ref{lem:B2}]
    1. Recall from Lemma \ref{lem:error-in-law}, estimate \eqref{ineq:error-in-law:W_TV(v,u)<1-epsilon} that 
    \begin{align*}
     &\W_{\TV}\big( \Law(v(t;v_0)),P_t(v_0,\cdot)  \big)\\ &\le 1-\frac{1}{2}\exp\Big\{-\frac{1}{2} \|\sigma^{-1}\|^2_{L(P_NH)}\lambda^2  \E\int_0^t\|u(s;u_0)-v(s;v_0)\|^2_H\d s\Big\}\\
     &\le 1-\frac{1}{2}\exp\Big\{- \|\sigma^{-1}\|^2_{L(P_NH)}\lambda^2  \|u_0-v_0\|_H^2  \frac{e^{\frac{1}{8}\beta C_1+\frac{1}{2}\beta\|u_0\|^2_H}}{C_2  - 4\frac{L}{\beta\gamma}-\frac{1}{8}\beta C_0}\Big\}.
\end{align*}
In the last implication, we employed estimate \eqref{ineq:E.int_0^t|u-v|^2_H}. It follows that for all $u_0,v_0\in B_R$, we have
\begin{align*}
     &\W_{\TV}\big( \Law(v(t;v_0)),P_t(v_0,\cdot)  \big)\\ &\le 1-\frac{1}{2}\exp\Big\{-4R^2 \|\sigma^{-1}\|^2_{L(P_NH)}\lambda^2  \frac{e^{\frac{1}{8}\beta C_1+\frac{1}{2}\beta R^2}}{C_2  - 4\frac{L}{\beta\gamma}-\frac{1}{8}\beta C_0}\Big\}=:1-\varepsilon.
\end{align*}
This establishes \eqref{ineq:B2:W_TV<1-epsilon}, as claimed.

2. Turning to \eqref{ineq:B2:E.theta(u,v)<e^(-ct)}, in light of \eqref{ineq:A:theta(v,u)<e^(-ct)theta(u_0,v_0)} from Lemma \ref{lem:A}, we infer for all $u_0,v_0\in B_R$ that
\begin{align*}
     \E\, \theta_\beta(u(t),v(t))  \le  \theta_{\beta/2}(u_0,v_0)\,r_1(t)&= \|u_0-v_0\|_H e^{\frac{1}{2}\beta \|u_0\|^2_H} r_1(t)\\
     &\le 2R e^{\frac{1}{2}\beta R}r_1(t),\quad t\ge T,
\end{align*}
where $T$ is the time constant as in \eqref{cond:t} and $r_1(t)$ is the function defined in \eqref{form:r_1(t)}. The proof is thus finished.
\end{proof}

\section{Acknowledgements} 
 The authors would like to thank anonymous referees for their providing a thorough review of this work. We appreciate their careful reading and insightful comments, which have improved the manuscript. Peng Gao would like to thank the financial support of the China Scholarship
Council (No. 202406620219). Peng Gao is supported by NSFC (Grant No. 12371188), 
Natural Science Foundation of Jilin Province (Grant No. YDZJ202201ZYTS306), and the Fundamental Research Funds for the Central Universities (Grant
No. 2412022ZD006).

\appendix

\section{Estimates on periodic functions} \label{sec:periodic_functions}
In this section, we state and prove useful inequalities on periodic functions through Lemma \ref{lem:R(u)>|D^2u|^2:antisymmetric} and Lemma \ref{lem:R(u)}. In turn, they were invoked to handle the difficulty of the nonlinearity with arbitrary diffusion parameter $\gamma$ in the stochastic KSE. The argument that we employ is closely adapting to the technique of \cite{collet1993global} dealing with the same issue in deterministic settings. Recalling $\A_L$, the space of odd, $L-$periodic function defined in \eqref{form:A_L}, the idea can be summarized as follows: we first construct an auxiliary function $\f\in \A_L$, which can be combined with $D^4u$ to create a dissipation effect for all $u\in \A_L\cap H^2_L$. This is presented in Lemma \ref{lem:R(u)>|D^2u|^2:antisymmetric}. Then, by modifying the function $\f$ appropriately, we extend the estimate of Lemma \ref{lem:R(u)>|D^2u|^2:antisymmetric} to cover any generic function $u$ that is not necessarily antisymmetric. The precise statement is given in Lemma \ref{lem:R(u)}.

\begin{lemma} \label{lem:R(u)>|D^2u|^2:antisymmetric}
 For all $L>0$, there exists a function $\f\in \A_L$ such that the following holds
    \begin{align} \label{ineq:R(u)>|D^2u|^2:antisymmetric}
       \gamma\|D^2 u\|^2_{H_L}-\|Du\|^2_{H_L}+\frac{1}{2}\la u^2,\f'\ra_{H_L}\ge \frac{1}{4}\gamma\|D^2u\|^2_{H_L} +\frac{1}{2}\|u\|^2_{H_L},\quad u\in \A_L\cap H^2_L.
    \end{align}
\end{lemma}

\begin{remark}
    We note that \cite[Proposition 2.1]{collet1993global} is a special case of Lemma \ref{lem:R(u)>|D^2u|^2:antisymmetric} for the instance $\gamma=1$. In order to establish estimate \eqref{ineq:R(u)>|D^2u|^2:antisymmetric} for any $\gamma>0$, we will slightly modify the proof of \cite[Proposition 2.1]{collet1993global}.
\end{remark}

\begin{proof}[Proof of Lemma \ref{lem:R(u)>|D^2u|^2:antisymmetric}]

Setting $q=\frac{2\pi}{L}$, we recast $u$ using the Fourier representation
\begin{align*}
    u(x)=\frac{\i}{\sqrt{L}}\sum_{n\in \zbb} u_n e^{\i nq x},\quad x\in\rbb.
\end{align*}
Since $u$ is an odd function, observe that
\begin{align*}
    u(-x) &= \frac{\i}{\sqrt{L}}\sum_{n\in\zbb} u_n[\cos(nqx)-\i\sin(nq x) ]= -\frac{\i}{\sqrt{L}}\sum_{n\in\zbb} u_n[\cos(nqx)+\i\sin(nq x) ].
\end{align*}
It follows that for all $n\ge0$
\begin{align*}
    u_{n}+u_{-n}=0,\quad\text{and}\quad u_0=0,
\end{align*}
whence, 
\begin{align*}
    u(x) = -\frac{2}{\sqrt{L}}\sum_{n\ge 1} u_n \sin(nqx).
\end{align*}
In particular, this implies that $u_n\in\rbb$. Likewise, given $\f\in\A_L$ with the representation
\begin{align*}
    \f(x)=\frac{\i}{\sqrt{L}}\sum_{n\in \zbb} \f_n e^{\i nq x},\quad x\in\rbb,
\end{align*}
it holds that $\f_n=-\f_{-n}\in\rbb$. Taking derivative on both sides produces
\begin{align*}
   \f'(x) = -\frac{1}{\sqrt{L}}\sum_{n\in\zbb} nq\f_n  e^{\i nqx}.
\end{align*}
Setting $\psi_n:=-nq\f_n$, we recast $\f'$ as
\begin{align} \label{form:phi'}
    \f'(x) = \frac{1}{\sqrt{L}}\sum_{n\in\zbb} \psi_n  e^{\i nqx},
\end{align}
and deduce 
\begin{align} \label{eqn:psi_n=psi_(-n)}
    \psi_n = \psi_{-n} \in\rbb,\quad\text{and}\quad \psi_0=0.
\end{align}

Next, considering the inner product $\la u^2,\f'\ra_{H_L}$, we have 
\begin{align*}
    \la u^2,\f'\ra_{H_L} & = \frac{1}{L\sqrt{L}}\sum_{n,k,\ell\in\zbb}\int_{-L/2}^{L/2} e^{\i (n+k+\ell)qx}\d x\, u_nu_k\psi_\ell\\
    & = \frac{1}{\sqrt{L}}\sum_{n+k+\ell=0}u_nu_k\psi_\ell .
\end{align*}
We employ the identities $u_n=-u_{-n}$, and $\psi_n=\psi_{-n}$, $\psi_0=0$ from \eqref{eqn:psi_n=psi_(-n)} and obtain further that
\begin{align*}
   \sum_{n+k+\ell=0}u_nu_k\psi_\ell =\sum_{n,k\in\zbb}u_nu_k\psi_{-(k+n)}&=\sum_{n,k\in\zbb}u_nu_k\psi_{|k+n|}  \\
   &=2\sum_{n,k\ge 1} u_nu_k\big(\psi_{k+n}-\psi_{|k-n|}\big)\\
  &= 2\sum_{n\ge 1}u_n^2\psi_{2n}+4\sum_{k>n\ge 1} u_nu_k\big(\psi_{k+n}-\psi_{k-n}\big),
\end{align*}
implying
\begin{align*}
   \sqrt{L}\la u^2,\f'\ra_{H_L}  =   2\sum_{n\ge 1}u_n^2\psi_{2n}+4\sum_{k>n\ge 1} u_nu_k\big(\psi_{k+n}-\psi_{k-n}\big).
\end{align*}
Altogether, we get
\begin{align} \label{eqn:R(u)=I_1+I_2}
     &  \gamma\|D^2 u\|^2_{H_L}-\|Du\|^2_{H_L}+\frac{1}{2}\la u^2,\f'\ra_{H_L} \notag \\
    &=\sum_{n\ge 1}u_n^2\Big( 2\gamma |nq|^4-2|nq|^2+\frac{1}{\sqrt{L}}\psi_{2n}  \Big) +\frac{2}{\sqrt{L}}\sum_{k>n\ge 1}u_nu_k\big(\psi_{k+n}-\psi_{k-n}\big) \notag \\
    & = I_1+I_2.
\end{align}
We emphasize that the above identity holds for all $\varphi\in \A_L$ where $\psi_n$ is as in \eqref{form:phi'}-\eqref{eqn:psi_n=psi_(-n)}. Now, we claim that there exists a choice of $\varphi\sim\{\varphi_n\}_n$ that verifies \eqref{ineq:R(u)>|D^2u|^2:antisymmetric}. For this, it will
be sufficient to construct $\f'\sim\{\psi_n\}_n$. To see this, we fix a positive integer $M=M(\gamma)$ large enough such that
\begin{align} \label{ineq:gamma|nq|^4>2|nq|^2+1}
    \gamma|nq|^4>2|nq|^2+1,\quad n>M.
\end{align}
Also, let $f:\rbb\to[0,\infty)$ be a smooth cut-off function satisfying 
\begin{align*}
    f(x) = \begin{cases}
        1,& |x|\le 1,\\
        \text{monotonicity},& 1\le |x|\le 2,\\
        0,& |x|\ge 2.
    \end{cases}
\end{align*}
With the above choices of $M$ and $f$, we define 
\begin{align*} 
    \psi_n = \sqrt{L}\Big(\frac{1}{\gamma}+1\Big)f\Big( \frac{n}{2M}\Big) .
\end{align*}
Concerning $I_1$ on the right-hand side of \eqref{eqn:R(u)=I_1+I_2}, we note that $\psi_{2n}$ satisfies
\begin{align} \label{eqn:psi_2n}
    \frac{1}{\sqrt{L}}\psi_{2n} = \begin{cases}
        \frac{1}{\gamma}+1,& 1\le n\le M,\\
        \ge 0,& n>M.
    \end{cases} 
    \end{align}
Observe that by \eqref{ineq:gamma|nq|^4>2|nq|^2+1} and \eqref{eqn:psi_2n}, the following holds 
\begin{align*}
    2\gamma |nq|^4-2|nq|^2+\frac{1}{\sqrt{L}}\psi_{2n}  \ge \gamma|qn|^4+1,\quad n\ge 1,
\end{align*}
whence
\begin{align} \label{ineq:R(u):I_1}
    I_1 \ge \sum_{n\ge 1}(\gamma|nq|^4+1)u_n^2 =\frac{1}{2} \big(\gamma\|D^2u\|^2_{H_L}+\|u\|^2_{H_L}\big) .
\end{align}
With regard to $I_2$ on the right-hand side of \eqref{eqn:R(u)=I_1+I_2}, we claim that 
\begin{align} \label{ineq:|psi_(k+n)-psi_(k-n)|}
     \frac{1}{\sqrt{L}}|\psi_{k+n}-\psi_{k-n}| & \le |f'|_\infty \Big(\frac{1}{\gamma}+1\Big)\frac{n}{M}.
\end{align}
Indeed, on the one hand, if $k+n\le 2M$,
\begin{align*}
    \psi_{k+n}-\psi_{k-n}=0.
\end{align*}
On the other hand, by the Mean Value Theorem, if $k+n>k-n\ge 2M$,
\begin{align*}
   \frac{1}{\sqrt{L}} |\psi_{k+n}-\psi_{k-n}|\le|f'|_\infty  \Big(\frac{1}{\gamma}+1\Big) \Big(\frac{k+n}{2M}-\frac{k-n}{2M}\Big) =  |f'|_\infty\Big(\frac{1}{\gamma}+1\Big)\frac{n}{M}.
\end{align*}
Otherwise, if $k+n>  2M \ge k-n$,
\begin{align*}
   \frac{1}{\sqrt{L}} |\psi_{k+n}-\psi_{k-n}| & =  \Big(\frac{1}{\gamma}+1\Big)\Big( f\Big( \frac{k+n}{2M} \Big)-f(1)\Big) \\
    &\le  |f'|_\infty \Big(\frac{1}{\gamma}+1\Big) \Big(\frac{k+n}{2M}-1\Big) \le |f'|_\infty \Big(\frac{1}{\gamma}+1\Big)\frac{n}{M}.
\end{align*}
In the last inequality above, we employed the fact that $k\le n+2M$. Turning back to $I_2$, we employ Cauchy-schwarz's inequality together with \eqref{ineq:|psi_(k+n)-psi_(k-n)|} to infer
\begin{align*}
    \frac{1}{2}I_2= \frac{1}{\sqrt{L}}\sum_{k>n\ge 1}u_nu_k\big(\psi_{k+n}-\psi_{k-n}\big) & \le  |f'|_\infty \Big(\frac{1}{\gamma}+1\Big)\frac{1}{M} \sum_{k>n\ge 1}nu_nu_k\\
    &\le  |f'|_\infty \Big(\frac{1}{\gamma}+1\Big)\frac{1}{M}\sum_{n\ge 1}nu_n \sqrt{\sum_{k\ge 1}\frac{1}{k^4} \sum_{k\ge 1}k^4u_k^2}\\
    &\le  |f'|_\infty \Big(\frac{1}{\gamma}+1\Big)\frac{1}{M}\sqrt{\sum_{n\ge 1}\frac{1}{n^2}\sum_{n\ge 1}n^4u_n^2}\cdot  \sqrt{\sum_{k\ge 1}\frac{1}{k^4} \sum_{k\ge 1}k^4u_k^2}\\
    &=  |f'|_\infty \Big(\frac{1}{\gamma}+1\Big)\sqrt{\sum_{n\ge 1}\frac{1}{n^2}\sum_{n\ge 1}\frac{1}{n^4} }\cdot\frac{1}{Mq^4}\|D^2u\|^2_{H_L}.
\end{align*}
So, provided $M$ is sufficiently large such that
\begin{align*}
    M>\frac{8}{q^4\gamma} |f'|_\infty \Big(\frac{1}{\gamma}+1\Big)\sqrt{\sum_{n\ge 1}\frac{1}{n^2}\sum_{n\ge 1}\frac{1}{n^4} },
\end{align*}
we get
\begin{align} \label{ineq:R(u):I_2}
    |I_2| \le \frac{\gamma}{4}\|D^2u\|^2_{H_L}.
\end{align}
Now, from \eqref{eqn:R(u)=I_1+I_2}, \eqref{ineq:R(u):I_1} and \eqref{ineq:R(u):I_2}, we obtain
\begin{align*}
  \gamma\|D^2 u\|^2_{H_L}-\|Du\|^2_{H_L}+\frac{1}{2}\la u^2,\f'\ra_{H_L}\ge \frac{1}{4}\gamma\|D^2u\|^2_{H_L}+\frac{1}{2}\|u\|^2_{H_L}.
\end{align*}
This establishes \eqref{ineq:R(u)>|D^2u|^2:antisymmetric}, thereby finishing the proof.

\end{proof}

\begin{lemma} \label{lem:R(u)}
 For all $L>0$, there exists a function $\f\in \A_{2L}$ such that the following holds
    \begin{align} \label{ineq:R(u)>|D^2u|^2}
       &\gamma\|D^2 u\|^2_{H_{2L}}-\|Du\|^2_{H_{2L}}+\frac{1}{2}\la u^2,\f'(\cdot+b)\ra_{H_{2L}}  \notag  \\
       &\ge \frac{1}{4}\gamma\|D^2u\|^2_{H_{2L}} +\frac{1}{2}\|u\|^2_{H_{2L}}-\frac{1}{4L}\big|\la u,\f'(\cdot+b)\ra_{H_{2L}}  \big|^2,
    \end{align}
    for all $b\in\rbb$, $u\in H^2_L$.
\end{lemma}

\begin{remark} \label{remark:R(u)} We remark that estimate \eqref{ineq:R(u)>|D^2u|^2} is actually not new as it was mentioned in the argument of \cite[Proposition 5.1]{ferrario2008invariant}, albeit without a proof. In our work, we opt for presenting the rigorous proof of Lemma \ref{lem:R(u)} not only for the sake of completeness, but also for its crucial role in Section \ref{sec:moment-estimate}, where we collect useful moment bounds. In turn, they are employed in Section \ref{sec:mixing} to establish the main result on the exponential mixing of the KSE.
    
\end{remark}

\begin{proof}[Proof of Lemma \ref{lem:R(u)}] Let $\f\in \A_{2L}$ be the function from Lemma \ref{lem:R(u)>|D^2u|^2:antisymmetric}. Observe that estimate \eqref{ineq:R(u)>|D^2u|^2} is equivalent to
\begin{align*}
       &\gamma\|D^2 u (\cdot-b) \|^2_{H_{2L}}-\|Du(\cdot-b) \|^2_{H_{2L}}+\frac{1}{2}\la u(\cdot-b) ^2,\f'\ra_{H_{2L}}  \notag  \\
       &\ge \frac{1}{4}\gamma\|D^2u(\cdot-b) \|^2_{H_{2L}} +\frac{1}{2}\|u(\cdot-b) \|^2_{H_{2L}}-\frac{1}{4L}\big|\la u(\cdot-b) ,\f'\ra_{H_{2L}}  \big|^2.
    \end{align*}
Since $u(\cdot-b)$ is still an element of $H^2_L$, it therefore suffices to establish \eqref{ineq:R(u)>|D^2u|^2} for $b=0$ and for all $u\in H^2_L$. 

Now, setting 
\begin{align*}
u_A(x) = \frac{u(x)-u(-x)}{2},\quad u_S=\frac{u(x)+u(-x)}{2}-u(0),
\end{align*}
we note that $u_A$ is an odd function whereas $u_S$ is an even function with $u_S(0)=0$. Moreover, we may recast $u$ as
\begin{align*}
u=u(0)+u_A+u_S.
\end{align*}
By a routine calculation, it holds that
\begin{align*}
\|D^2u\|^2_{H_{2L}}= \|D^2u_A+D^2u_S\|^2_{H_{2L}} = \|D^2u_A\|^2_{H_{2L}}+ \|D^2u_S\|^2_{H_{2L}}, 
\end{align*}
where in the last implication, we employed the fact that 
$\la D^2 u_A,D^2 u_S\ra_{H_{2L}}=0$ since this is the inner product between an odd function and an even function. Likewise, 
\begin{align*}
\|Du\|^2_{H_{2L}}= \|Du_A\|^2_{H_{2L}}+ \|Du_S\|^2_{H_{2L}}.
\end{align*}
Concerning $\la u^2,\f'\ra_{H_{2L}}$, we invoke the fact that
\begin{align*}
\la u(0)^2,\f'\ra_{H_{2L}}=\la u_Au_S,\f'\ra_{H_{2L}}=0,
\end{align*}
to obtain the identity
\begin{align*}
\la u^2,\f'\ra_{H_{2L}} & = \la u_A^2,\f'\ra_{H_{2L}}+\la u_S^2,\f'\ra_{H_{2L}}+\la u(0)^2,\f'\ra_{H_{2L}}\\
&\qquad +2\la u_Au_S,\f'\ra_{H_{2L}}+2u(0)\la u_A,\f'\ra_{H_{2L}}+2u(0)\la u_S,\f'\ra_{H_{2L}}\\
&= \la u_A^2,\f'\ra_{H_{2L}}+\la u_S^2,\f'\ra_{H_{2L}}+2u(0)\la u,\f'\ra_{H_{2L}}.
\end{align*}
It follows that
\begin{align} \label{eqn:R(u)=u_A+u_S}
\gamma\|D^2 u\|^2_{H_{2L}}-\|Du\|^2_{H_{2L}}+\frac{1}{2}\la u^2,\f'\ra_{H_{2L}}  
&=\gamma\|D^2 u_A\|^2_{H_{2L}}-\|Du_A\|^2_{H_{2L}}+\frac{1}{2}\la u_A^2,\f'\ra_{H_{2L}}  \notag \\
&\qquad + \gamma\|D^2 u_S\|^2_{H_{2L}}-\|Du_S\|^2_{H_{2L}}+\frac{1}{2}\la u_S^2,\f'\ra_{H_{2L}} \notag \\
&\qquad+u(0)\la u,\f'\ra_{H_{2L}}.
\end{align}
From \eqref{ineq:R(u)>|D^2u|^2:antisymmetric}, we readily have
\begin{align*}
\gamma\|D^2 u_A\|^2_{H_{2L}}-\|Du_A\|^2_{H_{2L}}+\frac{1}{2}\la u_A^2,\f'\ra_{H_{2L}}\ge \frac{1}{4}\gamma\|D^2u_A\|^2_{H_{2L}}+\frac{1}{2}\|u_A\|^2_{H_{2L}}.
\end{align*}
Turning to the terms involving $u_S$ on the right-hand side of \eqref{eqn:R(u)=u_A+u_S}, let $\Tcal$ be the operator defined as
\begin{align*}
\Tcal u (x) = u(x)\mathbf{1}\{x\ge 0\} - u(x) \mathbf{1}\{x< 0\}.
\end{align*}
Since $u_S$ is even with $u_S(0)=0$, observe that $\Tcal u_S$ is an odd and $2L$-periodic function with vanishing integral on $[-L,L]$. Also, a.e.
\begin{align*}
|D^m (\Tcal u_S)| = |D^m u_S|,\quad m=0,1,2.
\end{align*}
We invoke Lemma \ref{lem:R(u)>|D^2u|^2:antisymmetric} again to infer
\begin{align*}
  \gamma\|D^2 (\Tcal u_S) \|^2_{H_{2L}}-\|D(\Tcal u_S)\|^2_{H_{2L}}+\frac{1}{2}\la (\Tcal u_S)^2,\f'\ra_{H_{2L}}\ge \frac{1}{4}\gamma\|D^2(\Tcal u_S)\|^2_{H_{2L}}+\frac{1}{2}\|\Tcal u_S\|^2_{H_{2L}},
\end{align*}
whence
\begin{align*}
\gamma\|D^2 u_S\|^2_{H_{2L}}-\|Du_S\|^2_{H_{2L}}+\frac{1}{2}\la u_S^2,\f'\ra_{H_{2L}}\ge \frac{1}{4}\gamma\|D^2u_S\|^2_{H_{2L}}+\frac{1}{2}\|u_S\|^2_{H_{2L}}.
\end{align*}
So, from \eqref{eqn:R(u)=u_A+u_S}, we get
\begin{align*}
&\gamma\|D^2 u\|^2_{H_{2L}}-\|Du\|^2_{H_{2L}}+\frac{1}{2}\la u^2,\f'\ra_{H_{2L}}  \\
& \ge \frac{1}{4}\gamma\|D^2u_A\|^2_{H_{2L}}+\frac{1}{2}\|u_A\|^2_{H_{2L}}+\frac{1}{4}\gamma\|D^2u_S\|^2_{H_{2L}}+\frac{1}{2}\|u_S\|^2_{H_{2L}}+u(0)\la u^2,\f'\ra_{H_{2L}}\\
&= \frac{1}{4}\gamma\|D^2u\|^2_{H_{2L}}+\frac{1}{2}\|u_A+u_S\|^2_{H_{2L}}+u(0)\la u^2,\f'\ra_{H_{2L}}.
\end{align*}
Note that since $u$ has vanishing integral, we have
\begin{align*}
\|u_A+u_S\|^2_{H_{2L}} = \|u-u(0)\|^2_{H_{2L}}&=\|u\|^2_{H_{2L}}-2\la u(0),u\ra_{H_{2L}}+2Lu(0)^2\\
&=\|u\|^2_{H_{2L}}+2Lu(0)^2,
\end{align*}
whence
\begin{align*}
&\gamma\|D^2 u\|^2_{H_{2L}}-\|Du\|^2_{H_{2L}}+\frac{1}{2}\la u^2,\f'\ra_{H_{2L}}  \\
&\ge \frac{1}{4}\gamma\|D^2u\|^2_{H_{2L}}+\frac{1}{2}\|u\|^2_{H_{2L}}+Lu(0)^2+u(0)\la u^2,\f'\ra_{H_{2L}}\\
&\ge \frac{1}{4}\gamma\|D^2u\|^2_{H_{2L}}+\frac{1}{2}\|u\|^2_{H_{2L}}-\frac{1}{4L}\big|\la u^2,\f'\ra_{H_{2L}}\big|^2.
\end{align*}
This produces \eqref{ineq:R(u)>|D^2u|^2} for $b=0$ and for all $u\in H^2_L$. The proof is thus finished.
\end{proof}

\section{Auxiliary estimates} \label{sec:aux}

In this section, we collect a variety of estimates that are used to prove the main mixing theorem. We start with Lemma \ref{lem:|u|_infty<L|Du|} relating $L^\infty$ norm with $H^1$ norm. In particular, the result of Lemma \ref{lem:|u|_infty<L|Du|} appeared in the proof of Lemma \ref{lem:coupling}.

\begin{lemma} \label{lem:|u|_infty<L|Du|}
    Suppose that $u\in H^1$. Then,
    \begin{align} \label{ineq:|u|_infty<L|Du|}
        \|u\|_{L^\infty} \le \sqrt{L}\|Du\|_H.
    \end{align}
\end{lemma}
\begin{proof}
    Since $u$ is continuous and satisfies zero mean integral on $[-L/2,L/2]$, there must exist $x_0\in [-L/2,L/2]$ such that $u(x_0)=0$. We invoke Holder inequality to deduce for all $x\in[-L/2,L/2] $
    \begin{align*}
        |u(x)| = |u(x)-u(x_0)|= \Big|\int_{x_0}^x Du(s)\d s\Big|\le \sqrt{L}\sqrt{\int_{-L/2}^{L/2}|Du(s)|^2\d s}.
    \end{align*}
    This immediately produces \eqref{ineq:|u|_infty<L|Du|}, as claimed.
\end{proof}

Next, we assert an elementary inequality through Lemma \ref{lem:e^u<e^(-t+u)} which was employed in the proof of Lemma \ref{lem:Lyapunov} establishing Lypaunov functions.

\begin{lemma} \label{lem:e^u<e^(-t+u)}
    Suppose that $f:[0,\infty)\times \Omega\to[0,\infty)$ satisfies $f(0)\ge 0$ is deterministic and that
    \begin{align} \label{cond:e^u<e^(-t+u)}
       \E e^{f(t)} \le c_1  e^{c_2e^{-a t}f(0)},\quad t\ge 0,
    \end{align}
    for some positive constants $c_1$, $c_2$ and $a$ independent of $t$. Then, the following holds
    \begin{align} \label{ineq:e^u<e^(-t+u)}
        \E\,e^{f(t)} \le 2c_1\,e^{-\frac{1}{2} a t}\,e^{f(0)}+2c_1,\quad t\ge \frac{2}{a}\log (c_2+1).
    \end{align}
   
\end{lemma}
\begin{proof} We note that for all $t\ge \frac{2}{a}\log (c_2+1)$, it holds that
\begin{align*}
    c_2< e^{\frac{1}{2}at}.
\end{align*}
So, setting $p=e^{at}/c_2>1$ and $q=1/(1-c_2 e^{-at})>1$, we apply Young's inequality to the right hand side of \eqref{cond:e^u<e^(-t+u)} to infer
\begin{align*}
    e^{c_2e^{-a t}f(0)} \le c_2 e^{-at}e^{f(0)}+ 1-c_2 e^{-at} &=  c_2 e^{-\frac{1}{2}at}\cdot e^{-\frac{1}{2}at}e^{f(0)}+ 1-c_2 e^{-at} \\
    &\le  e^{-\frac{1}{2}at}e^{f(0)}+ 1.
\end{align*}
This immediately implies \eqref{ineq:e^u<e^(-t+u)}, as claimed.
\end{proof}

We conclude this section with two auxiliary estimates on Wasserstein distances in Lemma \ref{lem:W_(d_K,beta):triangle} and \ref{lem:W_(d_K,beta)<W_TV}. We have employed their results to establish the contracting properties in Lemma \ref{lem:contracting} and \ref{lem:d-small}.

\begin{lemma} \label{lem:W_(d_K,beta):triangle}
    Given $K,\beta>0$, let $d_{K,\beta}$ be the distance defined in \eqref{form:d_K}. Then, the following holds
    \begin{align} \label{ineq:W_(d_K,beta):triangle}
        \W_{d_{K,\beta}}(\nu_1,\nu_2)\le e^{2\beta/K^2} \big[\W_{d_{K,2\beta}}(\nu_1,\nu_3)+ \W_{d_{K,\beta}}(\nu_3,\nu_2)\big],\quad \nu_1,\nu_2,\nu_3\in \Pcal r(H).
    \end{align}
\end{lemma}
\begin{proof}
First of all, we claim that for all $u,v,z\in H$, 
\begin{align} \label{ineq:d_K,beta:triangle}
    d_{K,\beta}(u,v) \le e^{2\beta/K^2}\big[ d_{K,2\beta}(u,z)+ d_{K,2\beta}(z,v)\big].
\end{align}
Indeed, recalling expression \eqref{form:d_K} of $ d_{K,\beta}$, there are two cases to be considered depending on the values of $\|u-z\|_H$ and $\|z-v\|_H$. 

Case 1: $\max\{K\|u-z\|_H,K\|z-v\|_H\}\ge 1$. Without loss of generality, assume that $K\|u-z\|_H\ge 1$. Then, 
\begin{align*}
    K\theta_{2\beta}(u,z)\wedge K\theta_{2\beta}(z,u)= K\|u-z\|_H\big(e^{2\beta\|u\|^2_H} \wedge  e^{2\beta\|z\|^2_H}\big)\ge 1,
\end{align*}
whence,
\begin{align*}
    d_{K,2\beta}(u,z) = K\theta_{2\beta}(u,z)\wedge K\theta_{2\beta}(z,u)\wedge 1 =1.
\end{align*}
This immediately implies \eqref{ineq:d_K,beta:triangle} since $d_{K,\beta}(u,v)\le 1$.

Case 2: $\max\{K\|u-z\|_H,K\|z-v\|_H\}< 1$. In this case, observe that
\begin{align*}
    e^{\beta\|u\|^2_H} \le e^{2\beta\|u-z\|^2_H}e^{2\beta\|z\|^2_H}\le e^{2\beta/K^2}e^{\beta\|z\|^2_H}.
\end{align*}
Likewise, 
\begin{align*}
     e^{\beta\|v\|^2_H} \le e^{2\beta/K^2}e^{\beta\|z\|^2_H}.
\end{align*}
As a consequence, it holds that
\begin{align*}
    e^{\beta\|u\|^2_H}\wedge e^{\beta\|v\|^2_H}  \le e^{\beta\|u\|^2_H} \le e^{2\beta/K^2}\big(e^{2\beta\|u\|^2_H}\wedge e^{2\beta\|z\|^2_H} \big),
\end{align*}
and that
\begin{align*}
    e^{\beta\|u\|^2_H}\wedge e^{\beta\|v\|^2_H}  \le e^{\beta\|v\|^2_H} \le e^{2\beta/K^2}\big(e^{2\beta\|v\|^2_H}\wedge e^{2\beta\|z\|^2_H} \big).
\end{align*}
Now, using triangle inequality, we have
\begin{align*}
  & K\theta_{\beta}(u,v)\wedge K\theta_{\beta}(v,u) \\
  & = K\|u-v\|_H \big( e^{\beta\|u\|^2_H}\wedge e^{\beta\|v\|^2_H} \big)\\
   &\le K\|u-z\|_H \big( e^{\beta\|u\|^2_H}\wedge e^{\beta\|v\|^2_H} \big)+K\|z-v\|_H \big( e^{\beta\|u\|^2_H}\wedge e^{\beta\|v\|^2_H} \big)\\
   &\le e^{2\beta/K^2} K\|u-z\|_H \big( e^{2\beta\|u\|^2_H}\wedge e^{2\beta\|z\|^2_H} \big)+ e^{2\beta/K^2}K\|z-v\|_H \big( e^{2\beta\|z\|^2_H}\wedge e^{\beta\|v\|^2_H} \big)\\
   &= e^{2\beta/K^2}\big[ K\theta_{2\beta} (u,z)\wedge K\theta_{2\beta}(z,u)  +  K\theta_{2\beta} (v,z)\wedge K\theta_{2\beta}(z,v) \big].
\end{align*}
This implies that
\begin{align*}
    d_{K,\beta}(u,v)& = K\theta_{\beta}(u,v)\wedge K\theta_{\beta}(v,u) \wedge 1\\
    &\le e^{2\beta/K^2}\big[ K\theta_{2\beta} (u,z)\wedge K\theta_{2\beta}(z,u) \wedge 1 +  K\theta_{2\beta} (v,z)\wedge K\theta_{2\beta}(z,v) \wedge 1\big]\\
    &= e^{2\beta/K^2}\big[ d_{K,2\beta}(u,z)+ d_{K,2\beta}(z,v)\big],
\end{align*}
which proves \eqref{ineq:d_K,beta:triangle}, as claimed.

Turning back to \eqref{ineq:W_(d_K,beta):triangle}, we invoke \eqref{ineq:d_K,beta:triangle} to infer for all random variables $X_i\sim\nu_i$, $i=1,2,3$,
\begin{align*}
    \E\, d_{K,\beta}(X_1,X_2) \le  e^{2\beta/K^2}\big[ \E\,d_{K,2\beta}(X_1,X_3)+\E\, d_{K,2\beta}(X_3,X_2)\big].
\end{align*}
In view of the definition \eqref{form:W} of $\W_{d_{k,\beta}}$, we deduce 
\begin{align*}
    \W_{d_{K,\beta}}(\nu_1,\nu_2)\le e^{2\beta/K^2} \big[\W_{d_{K,2\beta}}(\nu_1,\nu_3)+ \W_{d_{K,\beta}}(\nu_3,\nu_2)\big],
\end{align*}
thereby finishing the proof. 
\end{proof}

\begin{lemma} \label{lem:W_(d_K,beta)<W_TV}
    Given $K,\beta>0$, let $d_{K,\beta}$ be the distance defined in \eqref{form:d_K}. Then, the following holds
    \begin{align} \label{ineq:W_(d_K,beta)<W_TV}
        \W_{d_{K,\beta}}(\nu_1,\nu_2)\le \W_{\TV}(\nu_1,\nu_2),\quad \nu_1,\nu_2\in \Pcal r(H).
    \end{align}
\end{lemma}
\begin{proof}
    We note that for all $u,v\in H$,\begin{align*}
        K\|u-v\|_He^{\beta\|u\|^2_H} \wedge K\|u-v\|_He^{\beta\|v\|^2_H}  \wedge 1\le \mathbf{1}\{u\neq v\}.
    \end{align*}
    In other words,
    \begin{align*}
        d_{K,\beta}(u,v) \le  \mathbf{1}\{u\neq v\}.
    \end{align*}
    This immediately produces \eqref{ineq:W_(d_K,beta)<W_TV} by virtue of expression \eqref{form:W}.
\end{proof}

\bibliographystyle{abbrv}
{\footnotesize\bibliography{wave-bib}}

\end{document}